\documentclass[11pt]{amsart}

\usepackage{graphicx}
\usepackage{latexsym}    % use all LaTeX fonts
\usepackage{amssymb}   % use all AMS fonts
\usepackage{amsmath}    % include some AMS-LaTeX functionality
\usepackage{amsbsy}
\usepackage{amsthm}
\usepackage{amsgen}
\usepackage{amsfonts}
\usepackage{enumitem}
\usepackage{array}
\usepackage[all]{xy}    % Xy-Pic for diagrams  (disabled)
\usepackage{tikz}
\usetikzlibrary{decorations.pathmorphing}
\usepackage{hyperref}
\usepackage{color}
\usepackage{verbatim}
\usepackage{url}

\tikzset{snake it/.style={decorate, decoration=snake}}

\newtheorem{thm}{Theorem}[section]
\newtheorem{cor}[thm]{Corollary}
\newtheorem{lem}[thm]{Lemma}
\newtheorem{prop}[thm]{Proposition}

\newtheorem{theorem}{Theorem}

\newtheorem{corollary}[theorem]{Corollary}

\newtheorem*{conjecture}{Rigidity Conjecture}

\theoremstyle{definition}
\newtheorem{defn}[thm]{Definition}

\newtheorem*{conj*}{Conjecture}

\newtheorem{rem}[thm]{Remark}

\newtheorem*{ack}{Acknowledgments}

\numberwithin{equation}{section}

\DeclareMathOperator{\Span}{span}
\DeclareMathOperator{\im}{Im}
\newcommand{\sph}{\mathrm{\mathbb{S}}}
\newcommand{\bq}{/ \hspace{-.15cm} /}
\newcommand{\Z}{\mathbb{Z}}
\newcommand{\C}{\mathbb{C}}
\def\N{\mathbb{N}}
\newcommand{\x}{\times}
\newcommand{\ox}{\otimes}
\newcommand{\vphi}{\varphi}
\newcommand{\id}{\mathrm{id}}

\def\bpm{\begin{pmatrix}}
\def\epm{\end{pmatrix}}
\def\bvm{\begin{vmatrix}}
\def\evm{\end{vmatrix}}
\def\bsm{\left(\begin{smallmatrix}}
\def\esm{\end{smallmatrix}\right)}
\def\beq{\begin{equation}}
\def\eeq{\end{equation}}

\DeclareMathOperator{\rank}{rank}
\DeclareMathOperator{\rk}{rk}

\newcommand{\CP}{\mathbb{CP}}

\newcommand{\Q}{\mathbb{Q}}

\newcommand{\SU}{\mathsf{SU}}

\def\ol{\overline}

\def\In{\subseteq}

\def\HH{\mathbb{H}}

\def\QQ{\mathbb{Q}}

\begin{document}

%-----------------------------------------------------------------------------
% BEGIN FRONT MATTER ----------------------------------------------------------
%-----------------------------------------------------------------------------

% TITLE

\title[Torus actions on rationally elliptic manifolds]{Torus actions on rationally elliptic manifolds}

% AUTHOR 1

\author[F.~Galaz-Garc\'ia]{F.~Galaz-Garc\'ia$^{*\dagger}$}
\address[Galaz-Garc\'ia]{Department of Mathematical Sciences, Durham University, United Kingdom.}
\email{fernando.galaz-garcia@durham.ac.uk}
\thanks{$^*$Received support from  SFB 878: \emph{Groups, Geometry \& Actions} at WWU M\"unster.}
\thanks{$^\dagger$Received support from the DFG grants GA 2050/2-1, SPP2026 ``Geometry at Infinity'' and 281869850, RTG 2229 ``Asymptotic Invariants and Limits of Groups and Spaces'').}

% AUTHOR 2

\author[M.~Kerin]{M.~Kerin$^{*\ddag}$}
\address[Kerin]{School of Mathematics, Statistics and Applied Mathematics, NUI Galway, Ireland.}
\email{martin.kerin@nuigalway.ie}
\thanks{$^\ddag$Received support from the DFG grant KE 2248/1-1, SPP2026 ``Geometry at Infinity''.}

% AUTHOR 3

\author[M.~Radeschi]{M.~Radeschi$^{*}$}
\address[Radeschi]{Department of Mathematics, University of Notre Dame, USA.}
\email{mradesch@nd.edu}
\thanks{}

% DATE
\date{\today}

% MATH SUBJECT CLASSIFICATION AND KEYWORDS


% ABSTRACT

\begin{abstract}  An upper bound is obtained on the rank of a torus which can act smoothly and effectively on a smooth, closed (simply connected) rationally elliptic manifold.  In the maximal-rank case, the manifolds admitting such actions are classified up to equivariant rational homotopy equivalence.
 \end{abstract}

\maketitle

%-----------------------------------------------------------------------------
% END FRONT MATTER ----------------------------------------------------------
%-----------------------------------------------------------------------------

%-----------------------------------------------------------------------------
%	MAIN MATTER	-----------------------------------------------------------------------------
%-----------------------------------------------------------------------------

%---------------------------------------------
% SECTION: INTRODUCTION
%---------------------------------------------

\section{Introduction}

Recall that a simply connected topological space $X$ is \emph{rationally elliptic} if $\dim_\Q H^*(X; \Q) < \infty$ and $\dim_\Q (\pi_*(X) \ox \Q) < \infty$.  An action of a compact Lie group $G$ on $X$ is said to be \emph{effective} if $g=e \in G$ whenever $g \cdot x=x$ for all $x\in X$.  The action is \emph{almost free} if, for every $x \in X$, the isotropy group $G_x=\{g \in G\ | \ g \cdot x=x \}$ is finite.  

% THM: RANK BOUNDS

\begin{theorem}
\label{T:RANK_BOUND}
Let $M^n$ be a smooth, closed, (simply connected) rationally elliptic $n$-dimensional manifold equipped with a smooth, effective action of the $k$-torus $T^k$.  Then $k\leq \left\lfloor \frac{2n}{3} \right\rfloor$.  Moreover, if the action is almost free, then $k \leq \left\lfloor \frac{n}{3} \right\rfloor$.
\end{theorem}

To the best of the authors' knowledge, these simple inequalities have not appeared in the literature, even though torus actions on rationally elliptic spaces have received much attention (see, for example, \cite{Al, Ha} and related papers).  In the equality cases, it is possible to determine which (equivariant) rational homotopy types can arise.  For a definition of equivariant rational homotopy equivalence, see Definition \ref{D:ERH}.

% THM B: RIGIDITY

\begin{theorem}
\label{T:RIGIDITY}
Let $M^n$, $n \geq 3$, be an $n$-dimensional, smooth, closed, (simply connected) rationally elliptic manifold equipped with a smooth, effective action of the $k$-torus $T^k$, $k\geq 1$.
\begin{enumerate}
	\item \label{L:AF} 
	If $T^k$ acts almost freely and $k = \left\lfloor \frac{n}{3} \right\rfloor$, then $M^n$ is rationally homotopy equivalent to a product $X \x \prod^{k-1}_{i=1} \sph^3$, where $X \in \{ \sph^3,\sph^2 \x \sph^3, \sph^5\}$.
	\item If  $k=\left\lfloor\frac{2n}{3}\right\rfloor$, then $M^n$ is rationally homotopy equivalent to a product $N^m \x \prod_{i=1}^{d} \sph^3$, where $m \in \{3,4,5,7,10\}$, $n = 3d + m$ and
	$$
		N^m = 
			\begin{cases}
			\sph^3, 													& \text{if } m = 3;\\
			\sph^4,\ \CP^2,\ \sph^2\x\sph^2,\text{ or } \CP^2\#\CP^2,				& \text{if } m = 4;\\
			\sph^2\x\sph^3\text{ or } \sph^5,								& \text{if } m = 5;\\
			\sph^7,\ \sph^2\times\sph^5 \text{ or }T^1(\sph^2\times\sph^2), 			& \text{if } m = 7;\\
			\sph^5  \x \sph^5,	 										& \text{if } m = 10.\\
			\end{cases}  
	$$
	Here $T^1(\sph^2\times\sph^2)$ denotes the unit tangent bundle of $\sph^2\times\sph^2$. Each manifold $N^m \x \prod_{i=1}^{d} \sph^3$ is equipped with a canonical linear $T^k$ action such that the rational homotopy equivalence is $T^k$-equivariant (in the sense of Definition~\ref{D:ERH}).
	\end{enumerate}
\end{theorem}

It is easy to see that each of the model spaces in Theorem \ref{T:RIGIDITY} admits a maximal-rank torus action of the appropriate type. In the effective case, the rigidity part is obtained in two steps. First, it is shown that any manifold satisfying the hypotheses of %in 
part (2) of Theorem~\ref{T:RIGIDITY} must be (equivariantly) rationally homotopy equivalent to a manifold of one of the following forms:
\begin{enumerate}
\item $X \x \prod \sph^3$,  with $X \in \{\sph^3, \sph^4, \sph^5, \sph^7, \sph^5 \x \sph^5\}$; \vspace*{1mm}
\item $(Y \x \prod \sph^3)/\sph^1$,  with $Y \in \{\sph^3, \sph^5\}$; or \vspace*{1mm}
\item $(\prod \sph^3)/T^2$.
\end{enumerate}
The second step is to show that any manifold of this form belongs to one of the finitely many options listed in Theorem \ref{T:RIGIDITY}. The biggest difficulty is to classify the rational homotopy types of manifolds of the form $(\prod \sph^3)/T^2$ and is dealt with in Theorem \ref{T:biqs}.

The conclusion of Theorem \ref{T:RIGIDITY} regarding finitely many rational homotopy types in each dimension is in contrast to the case of effective actions of rank $k = \left\lfloor \frac{2n}{3} \right\rfloor  - 1$, even in low dimensions. For example, B.\ Totaro \cite{To} has demonstrated that there are infinitely many rational homotopy types of $6$-dimensional manifolds of the form $(\sph^3 \x \sph^3 \x\sph^3) / T^3$, each of which admits an effective $T^3$ action.  Similarly, in each dimension $n = 3m + 1$, $m \not\equiv 1 \mod 4$, there are infinitely many rational homotopy types of manifolds which admit an almost-free torus action of rank $\left\lfloor\frac{n}{3}\right\rfloor - 1$ (see Proposition \ref{P:AMAF}).

% RIGIDITY CONJECTURE

It is natural to wonder whether the classifications in Theorem \ref{T:RIGIDITY} can be improved to (equivariant) homeomorphism or diffeomorphism.  

\begin{conjecture}
Let $M^n$, $n \geq 3$, be an $n$-dimensional, smooth, closed,  (simply connected) rationally elliptic manifold equipped with a smooth, effective action of the torus $T^k$ of rank $k = \left\lfloor\frac{2n}{3}\right\rfloor$.  Then $M^n$ is equivariantly diffeomorphic to an effective, linear action of $T^k$ on a manifold of one of the following forms:
\begin{enumerate}
\item $X \x \prod \sph^3$,  with $X \in \{\sph^3, \sph^4, \sph^5, \sph^5 \x \sph^5, \sph^7\}$;
\item $(Y \x \prod \sph^3)/\sph^1$,  with $Y \in \{\sph^3, \sph^5\}$; or
\item $(\prod \sph^3)/T^2$.
\end{enumerate}
\end{conjecture}

In low dimensions, it is possible to obtain some partial results in this direction.  These can be found in Section \ref{S:DIFFEO_CLASS}.

% MOTIVATION

\bigskip
The original motivation for the present work comes from the study of closed Riemannian manifolds with positive or non-negative sectional curvature. One of the central conjectures in the subject is the following:

% BOTT CONJECTURE

\begin{conj*}[Bott] A closed, simply connected manifold which admits a Riemannian metric of non-negative sectional curvature is rationally elliptic.
\end{conj*}

Although all manifolds known to admit positive or non-negative sectional curvature are rationally elliptic, examples of such manifolds are rare and difficult to find.  Nevertheless, Theorem \ref{T:RANK_BOUND} implies that a simply connected $n$-manifold admitting both a metric of non-negative curvature and an effective action by a torus of rank greater than $\left\lfloor\frac{2n}{3}\right\rfloor$ would be a counter-example to the Bott Conjecture.  On the other hand, in \cite{GGS} it is conjectured that $\left\lfloor\frac{2n}{3}\right\rfloor$ is the maximal rank of a torus which can act effectively and isometrically on a simply connected $n$-manifold with non-negative curvature.  Together with the Bott Conjecture, Theorem \ref{T:RIGIDITY} then provides further evidence for the expectation that there are strong restrictions on the topology of manifolds which admit a metric of non-negative curvature.

As it happens, the Bott Conjecture was verified in \cite{GKRW} in the presence of an effective, isometric torus action which is also slice maximal (see Section \ref{S:MaxEff} for a definition).  Theorems \ref{T:RANK_BOUND} and \ref{T:RIGIDITY} then yield the possible rational homotopy types of non-negatively curved, simply connected manifolds in the presence of a maximal-rank, effective, isometric, slice-maximal torus action, while C.\ Escher and C.\ Searle \cite{ES} have announced a proof of an analogue of the Rigidity Conjecture in this setting.

% TORAL RANK CONJECTURE

\bigskip
There is a further interesting consequence of Theorem \ref{T:RIGIDITY}.  Recall that the largest integer $r$ for which $M^n$ admits an almost-free $T^r$-action is called the \emph{toral rank} of $M^n$, and is denoted $\rk(M)$.  By Theorem \ref{T:RANK_BOUND}, it is clear that $\rk(M) \leq \left\lfloor \frac{n}{3} \right\rfloor$.  The Toral Rank Conjecture, formulated by S.\ Halperin, asserts that $\dim H^*(M ; \Q) \geq 2^{\rk(M)}$.

% COROLLARY

\begin{corollary} 
Let $M^n$ be a closed,  (simply connected) rationally elliptic, smooth $n$-manifold with a smooth, effective action of the $k$-torus $T^k$, $k\geq 1$.  If $k=\left\lfloor\frac{2n}{3}\right\rfloor$, or if $T^k$ is of rank $\left\lfloor\frac{n}{3}\right\rfloor$ and acts almost freely, then $M^n$ satisfies the Toral Rank Conjecture.
\end{corollary}

% PROOF 

\begin{proof} Let $T^r$ act almost freely on $M$.  Given $H^2(M;\Q) = \Q^{b_2(M)}$, there is a principal $T^{b_2(M)}$-bundle over $M$ with (rationally) $2$-connected, rationally elliptic total space $P$.  As any action by a torus $T$ on $M$ lifts to a $T \x T^{b_2(M)}$ action on $P$, the effective $T^k$ action (resp. the almost-free $T^r$ action) on $M$ lifts to an effective $T^k \x T^{b_2(M)}$ action (resp. an almost-free $T^r \x T^{b_2(M)}$ action) on $P$.  Moreover, observe from Theorem \ref{T:RIGIDITY} that $k=\left\lfloor\frac{2n}{3}\right\rfloor$ implies that $k + b_2(M)=\left\lfloor \frac{2(n + b_2(M))}{3} \right\rfloor$, while $k=\left\lfloor\frac{n}{3}\right\rfloor$ and the $T^k$ action being almost free implies that the lifted $T^k \x T^{b_2(M)}$ action is almost free of rank $k + b_2(M)=\left\lfloor \frac{n + b_2(M)}{3} \right\rfloor$.  Now, since $H^2(P;\Q) = 0$, Theorem \ref{T:RIGIDITY} yields that $P$ must have the rational cohomology of a product of spheres of dimension $\geq 3$.  By \cite[Prop.\ 7.23]{FOT}, $P$ satisfies the Toral Rank Conjecture, i.e. $H^*(P;\Q) \geq 2^{r + b_2(M)}$.  The result now follows from the observation that $\dim H^*(P ; \Q) \leq \dim H^*(T^{b_2(M)} ; \Q) \cdot \dim H^*(M ; \Q)$.
\end{proof}

Finally, note that all of the statements above for almost-free torus actions hold in the more general situation of (compact) rationally elliptic topological spaces of finite formal dimension, i.e.~without any smoothness assumptions whatsoever.  On the other hand, smoothness is required in the case of effective torus actions to ensure that the slice representation is linear and in order to apply the results of \cite{GKRW}.

The paper is organised as follows: In Section \ref{S:PRELIM} some basic definitions and facts about group actions and rational ellipticity are collected.  Section \ref{S:PROOF_THM_A} contains the proof of the inequalities in Theorem \ref{T:RANK_BOUND}, as well as some simple corollaries.  Sections \ref{S:PROOF_THM_B-(1)} and \ref{S:MaxEff} deal with the classification statements of Theorem \ref{T:RIGIDITY}.  The proof that only finitely many rational homotopy types arise in the classification can be found in Section \ref{S:MaxEff} (the case $b_2(M^n) = 1$) and in Section \ref{S:Biquotients} (the more difficult case of $b_2(M^n) = 2$).  Finally, in Section \ref{S:DIFFEO_CLASS}, some stronger classification results in low dimensions are discussed.

% ACKNOWLEDGEMENTS

\begin{ack} All three authors would like to express their gratitude to Michael Wiemeler, for numerous discussions, and to the referee, whose comments helped improve the paper.  The second-named author wishes to thank Anand Dessai for conversations which later proved useful in Section \ref{S:DIFFEO_CLASS}.
\end{ack}

% SECTION: PRELIMINARIES
%----------------------------------------
%----------------------------------------

\section{Preliminaries}
\label{S:PRELIM}

% SS: GP ACTIONS
%----------------------------

\subsection{Group actions} \hspace*{1mm}

Let $\Phi: G\x X \longrightarrow X$,  $(g,x) \mapsto g \cdot x$, be a continuous action by a compact Lie group $G$ on a topological space $X$.  Denote the orbit of a point $x\in X$ under the action of $G$ by $G \cdot x \cong G/G_x$, where $G_x = \{ g \in G \mid g \cdot x = x \}$ is the \emph{isotropy subgroup} of $G$ at $x$.  If the space $X$ is a smooth manifold and $\Phi$ is a smooth map, then the action is said to be \emph{smooth} and, in that case, the orbits are smooth submanifolds of $X$.

The action is \emph{effective} if the subgroup $\{g \in G \mid \Phi(g, \cdot) = \id_X \} \In G$ is trivial, and it is \emph{almost free} (resp.\ \emph{free}) if the isotropy subgroup $G_x$ is finite (resp.\ trivial) for all $x \in X$.  The \emph{orbit} or \emph{quotient space} of the action will be denoted by $X/G$.  If $X$ is a smooth manifold and $G$ acts freely (resp.\ almost freely) on $X$, then $X/G$ is a smooth manifold (resp.\ orbifold) of dimension $\dim(X) - \dim(G)$.

To every compact Lie group $G$ one can associate a contractible space $EG$ on which $G$ acts freely.  The quotient space $BG = EG/G$ is called the \emph{classifying space} of $G$ and the principal $G$-bundle $G \to EG \to BG$ is called the \emph{universal} $G$-\emph{bundle}.

Given the action $\Phi$ of $G$ on $X$ above, there is a fibre bundle 
$$
X \to X_G \to BG
$$ 
associated to the universal $G$-bundle, where $X_G = EG \x_G X = (EG \x X)/G$ is the quotient of $EG \x X$ by the (free) diagonal $G$ action.  The space $X_G$ is called the \emph{Borel construction} corresponding to the action $\Phi$.  Furthermore, as $EG$ is contractible, $EG \x X$ is homotopy equivalent to $X$ and the principal $G$-bundle $G \to EG \x X \to X_G$ yields, up to homotopy, a $G$-bundle 
$$
G \to X \to X_G.
$$

The \emph{equivariant cohomology} of $X$ with respect to the action $\Phi$ and with coefficients in a ring $R$ is given by $H^*_G(X ; R) = H^*(X_G; R)$, i.e. the ordinary $R$-cohomology of the Borel construction $X_G$.  In particular, if $X$ is compact and $G$ is a torus, then the action $\Phi$ is almost free if and only if the inequality $\dim_\Q H^*_G(X ; \Q) < \infty$ holds \cite[Prop.\ 4.1.7]{AP}.

\medskip

% SS: RHT
%----------------------------

\subsection{Rational  homotopy theory}\label{SS:RHT} \hspace*{1mm}

Below (with minor abuses of terminology) is a brief summary of those aspects of rational homotopy theory pertinent to the results on rationally elliptic manifolds in the present article.  A more complete treatment can be found in, for example, \cite{FHT, FOT}.  At the end, a new definition of equivariance for rational homotopy equivalence is introduced.

Let $X$ be a simply connected topological space.  The \emph{rational homotopy groups} of $X$ are given by the $\Q$-vector spaces $\pi_i(X) \ox \Q$, $i \in \N$, of dimension $d_i(X) = \dim_\Q (\pi_i(X) \ox \Q)$.  The space $X$ is said to be \emph{rationally elliptic} if 
$$
\dim_\Q H^*(X; \Q) < \infty \text{\ and\ } \dim_\Q (\pi_*(X) \ox \Q) = \sum_{i = 1}^\infty d_i(X) < \infty.
$$

Whenever $\dim_\Q H^*(X; \Q) < \infty$, there is an integer $n_X$, called the \emph{formal dimension} of $X$, such that $H^{n_X}(X; \Q)\neq 0$ and $H^i(X; \Q) = 0$, for all $i > n_X$.  If $X$ is a closed manifold, then clearly $n_X = \dim(X)$, since $\pi_1(X) = 0$.  The \emph{homotopy Euler characteristic} of a rationally elliptic space $X$ is given by 
\[
\chi_\pi (X) = \sum_{i  = 1}^{\infty} (-1)^i \, d_i(X)\,.
\]

As $X$ is simply connected, set $V^0 = \Q$ and $V^1 = \{0\}$.  From the rational homotopy groups, one can then construct a graded vector space $V_X = \oplus_{i = 0}^\infty V^i$ associated to $X$, where
\[
V^i \cong \textrm{Hom}(\pi_i(X),\QQ)\cong \pi_i(X) \ox \Q \cong \Q^{d_i(X)} \, ,\ \ i \geq 2.
\]
An element $v \in V^i$ is said to be \emph{homogeneous} of \emph{degree} $\deg(v) = i$.  

The tensor algebra $TV_X$ on $V_X$ has an associative multiplication with a unit $1 \in V^0$ given by the tensor product $V^i \ox V^j \to V^{i + j}$.  Taking the quotient of $TV_X$ by the ideal  generated by the elements $v \ox w - (-1)^{ij} w \ox v$, where $\deg(v) = i$, $\deg(w) = j$, yields the \emph{free commutative graded algebra} $\wedge V_X$.  In particular, multiplication in $\wedge V_X$ satisfies $v \cdot w = (-1)^{ij} w \cdot v$, for all $v \in V^i$ and $w \in V^j$.

Given a homogeneous basis $\{v_1, \dots, v_N \}$ of $V_X$, set $\wedge(v_1, \dots, v_N) = \wedge V_X$.  Moreover, denote the linear span of elements $v_{i_1} v_{i_2} \cdots v_{i_q} \in \wedge V_X$, $1 \leq i_1 \leq i_2 \leq \dots \leq i_q \leq N$, of word-length $q$ by $\wedge^q V_X$.  Define $\wedge^+ V_X = \oplus_{q \geq 1} \wedge^q V_X$.

As it turns out, $\wedge V_X$ possesses a linear \emph{differential} $d_X$, i.e. a linear map $d_X : \wedge V_X \to \wedge V_X$ satisfying the following properties:
\begin{enumerate}
\item \label{L:deg} 
$d_X$ has degree $+1$, i.e. $d_X$ maps elements of degree $i$ to elements of degree $i+1$.
\item $d_X^2 = 0$.
\item $d_X$ is a derivation, i.e. $d_X (v \cdot w) = d_X (v)\cdot w + (-1)^{\deg(v)} v \cdot d_X(w)$.
\item \label{L:nil}
$d_X$ is nilpotent, i.e. there is an increasing sequence of graded subspaces $V(0) \In V(1) \In \cdots $ such that $V = \cup_{k=0}^\infty V(k)$, $d_X|_{V(0)} \equiv 0$ and $d_X : V(k) \to \wedge V(k-1)$, for all $k \geq 1$.
\end{enumerate}
In addition, $d_X$ satisfies:
\begin{enumerate}[resume]
\item $d_X$ is decomposable, i.e. $\im(d_X) \In \wedge^{\geq 2} V_X$.
\end{enumerate}

Since $d_X$ is a derivation, it clearly depends only on its restriction to $V_X$.  The pair $(\wedge V_X, d_X)$ is called the \emph{minimal model} for $X$ and its corresponding (rational) cohomology satisfies $H^*(\wedge V_X, d_X) = H^*(X; \Q)$.

If $Y$ is another simply connected topological space, then $X$ and $Y$ are said to be \emph{rationally homotopy equivalent} (denoted $X \simeq_\Q Y$) if their minimal models are isomorphic, i.e. if there is a linear isomorphism $f:\wedge V_X \to \wedge V_Y$ which respects the grading and satisfies $f \circ d_X = d_Y \circ f$ and $f(v \cdot w) = f(v)\cdot f(w)$.  It is important to note that the isomorphism $f$ is not necessarily induced by a map between $X$ and $Y$.  In fact, $X \simeq_\Q Y$ if and only if there is a chain of maps $X \to Y_1 \leftarrow Y_2 \to \cdots \leftarrow Y_s \to Y$ such that the induced maps on rational cohomology are all isomorphisms.

Let now $E$ and $X$ be simply connected topological spaces and let $p: E \to X$ be a Serre fibration with simply connected fibre $F$.  If $(\wedge V_X, d_X)$ and $(\wedge V_F, d_F)$ are the minimal models of $X$ and $F$, respectively, then $E$ has a \emph{relative minimal model} of the form 
$$
(\wedge V_X \ox \wedge V_F, D) = (\wedge (V_X \oplus V_F), D),
$$ 
where $D|_{\wedge V_X} = d_X$ and $D(v) - d_F(v) \in \wedge^+ V_X \ox \wedge V_F$, for all $v \in V_F$.  Note that the relative minimal model $(\wedge V_X \ox \wedge V_F, D)$ need not be a minimal model for $E$ since, although the differential $D$ satisfies the conditions analogous to \eqref{L:deg}--\eqref{L:nil} above, it may not be decomposable.  Nevertheless, one still has $H^*(\wedge V_X \ox \wedge V_F, D) = H^*(E; \Q)$.

% PROP
% NOTE: The first two equations are independent of the existence of a torus action. 

\begin{prop}[\cite{FHT}, Chap.~32]
\label{P:ELLIPTIC_INEQ}
Let $X$ be a (simply connected) rationally elliptic topological space of formal dimension $n_X$.
Then:
\begin{align}
\label{E:rat-ell-ineq-even}	&n_X \geq \sum_{j=1}^\infty (2j) \, d_{2j}(X) \,; \\
\label{E:rat-ell-eq}			&n_X = \sum_{j=1}^\infty (2j+1) \, d_{2j+1}(X) - \sum_{j=1}^\infty (2j-1) \, d_{2j}(X).
\end{align}
Suppose further that $X$ admits an almost-free action by a torus of rank $k$.  Then
\begin{align}
\label{E:rat-ell-ineq-torus} 	&k \leq -\chi_\pi (X) \,.
\end{align}
\end{prop}

% LEM: BOREL CONSTRUCTION IS RAT ELLIPTIC
The following lemma is well known, but a proof is provided for completeness.

\begin{lem} 
\label{L:Borel_RE}
Assume that a $k$-dimensional torus $T^k$ acts almost freely on a compact, simply connected topological space $X$ of formal dimension $n$.  If $X$ is rationally elliptic, then the Borel construction $X_T$ is rationally elliptic and of formal dimension $n - k$.
\end{lem}

\begin{proof}
As previously mentioned, the inequality $\dim_\Q H^*(X_T ; \Q) < \infty$ follows from Proposition 4.1.7 in \cite{AP}.  Given this, the Serre spectral sequence of the (homotopy) fibration $T^k \to X \to X_T$  yields that the formal dimension of $X_T$ is $n - k$.  Therefore, it remains to show only that $\dim_\Q (\pi_*(X_T) \ox \Q) < \infty$.  As $X$ is rationally elliptic and $\pi_j(T^k) = 0$ for all $j \geq 2$, this follows immediately from the long exact sequence of homotopy groups for the fibration $T^k \to X \to X_T$.
\end{proof}

% Equivariance in RHT

The following definition gives a notion of equivariant rational homotopy equivalence.  In this article, it will be used in the context of torus actions.

\begin{defn} 
\label{D:ERH}
Let $X$ and $Y$ be simply connected topological spaces which both admit an effective action by a compact Lie group $G$.  A rational homotopy equivalence between $X$ and $Y$ is said to be \emph{$G$-equivariant} if the corresponding Borel constructions $X_G$ and $Y_G$ are also rationally homotopy equivalent and there exists a commutative diagram
\[
\xymatrix{
H^*(Y;\Q) \ar[r]& H^*(X;\Q) \\
H^*_{G}(Y,\Q) \ar[u]^{} \ar[r]&  H^*_{G}(X,\Q)\ar[u]^{}
}
\]
where the horizontal arrows are isomorphisms induced by the respective rational homotopy equivalences.
\end{defn}

% SECTION: PROOF OF THEOREM A    
%---------------------------------------------------
%---------------------------------------------------

\section{Bounds on the rank of a torus action}
\label{S:PROOF_THM_A}
Let  $M^n$ be an $n$-manifold which is smooth, closed,  simply connected and rationally elliptic, and on which  the $k$-torus $T^k$ acts smoothly and effectively.

% SS: PROOF OF PART (1)
%------------------------------------

\subsection*{Almost-free bound} Assume that $T^k$ acts on $M^n$ almost freely and let  $M_T$ be the corresponding Borel construction.  By Lemma~\ref{L:Borel_RE}, $M_T$ is rationally elliptic of formal dimension $n - k$. Therefore, by Proposition~\ref{P:ELLIPTIC_INEQ},
\begin{align}
\label{E:Homotop_ineq}
	n-k 	& \geq  	\sum_j(2j) \, d_{2j}(M_T)\notag \\%[.2cm]
		& \geq 	 2 \, d_2(M_T)\\%[.2cm]
		& \geq 	 2k. \notag
\end{align}
It now follows immediately that $3k\leq n$, i.e. $k \leq \left\lfloor \frac{n}{3} \right\rfloor$.

% REMARK

\begin{rem}
Observe that the argument to establish an upper bound on the rank of a torus acting almost freely goes through verbatim in the case of a rationally elliptic topological space $X$ of formal dimension $n$.  However, (local) smoothness of the space and action are needed to obtain the effective bound below.
\end{rem}

% NOTE: Reference for Hsiang's result: Theorem 7.7 of Felix, Oprea and Tanre. According to [FOT], the result appears in Hsiang's book "Cohomological theory of topological transformation groups". See also Allday & Puppe, prop. 4.1.7

% SS: PROOF OF PART (2)
%------------------------------------

\subsection*{Effective bound} If the $T^k$ action is only effective, let $s > 0$ be the dimension of the largest isotropy subgroup of the action. Since $M^n$ is compact, there exist only finitely many orbit types. By looking at the Lie algebra of $T^k$, it is clear that a subgroup $T^{k-s} \In T^k$ can be found, whose intersection with each isotropy group is finite. As a consequence, $T^{k-s}$ acts almost freely on $M^n$. The bound on the rank of almost-free actions established above then yields  $3(k-s)\leq n$. 

Suppose now that $p \in M^n$ is a point with isotropy subgroup $T_p$ of dimension $s$.  The orbit $T^k \! \cdot p$ through $p$ has dimension $k - s$, and the normal space $\nu_p(T^k \! \cdot p)$ at $p$ has dimension $n-k+s$.  The connected component of the identity in $T_p$ is a torus $T^s$ of rank $s$, which acts linearly and effectively on $\nu_p(T^k \! \cdot p)$.  Hence $2s \leq n - k + s$ or, equivalently, $s \leq n - k$.

Combining these two inequalities yields
\[
n \geq 3(k-s) \geq 3k - 3(n-k) = 6k - 3n,
\]
from which it follows $3k \leq 2n$.

% REMARK

\begin{rem} In establishing an upper bound on the rank of a torus acting effectively, the hypothesis that $M^n$ is rationally elliptic was used only to ensure that $3(k-s) \leq n$.  Even if this hypothesis is dropped, the inequality $s \leq n - k$ remains valid.  Therefore, if $3s \geq n$, one obtains $n \leq 3s \leq 3(n-k)$ and, consequently, $3k \leq 2n$.

In particular, to confirm the upper bound $\left\lfloor \frac{2n}{3} \right\rfloor$ on the symmetry rank of a non-negatively curved, simply connected $n$-manifold conjectured in \cite{GGS}, one need only show that $k \leq \left\lfloor \frac{2n}{3} \right\rfloor$ when $3s < n$, i.e. whenever the maximal dimension of an isotropy subgroup is small. C.\ Escher and C.\ Searle have independently made a similar observation in their preprint \cite{ES}.
\end{rem}

% COROLLARY
\bigskip
To finish this section, a number of simple applications of Theorem \ref{T:RANK_BOUND} are provided, the statements of which may be useful in their own right.

\begin{cor}
Let $M^n$ be a closed,  (simply connected) rationally elliptic, smooth $n$-manifold. If a torus $T^k$ acts smoothly on $M^n$ with cohomogeneity $d$, then $n\leq 3d$.
\end{cor}

% PROOF

\begin{proof}
Without loss of generality, it may be assumed that $T^k$ acts effectively on $M^n$, since the principal isotropy group fixes all of $M^n$ pointwise. It follows that $d=n-k$ and $3k\leq 2n = 3n-n$, whence $n\leq 3(n-k)=3d$.
\end{proof}

It was shown in \cite{GH} that a closed, smooth, simply connected manifold which admits a cohomogeneity-one action by a compact Lie group $G$ is rationally elliptic.  If one wishes to classify cohomogeneity-one manifolds, it is useful to be able to find restrictions on which Lie groups can arise.  

% COROLLARY

\begin{cor}
Let $M^n$ be a smooth, closed, simply connected $n$-manifold on which a compact Lie group $G$ acts effectively and smoothly with cohomogeneity one.  Then $3\rank(G)\leq 2n$.
\end{cor}

% PROOF

\begin{proof}
By considering the action on $M^n$ of the maximal torus inside $G$, the result follows immediately from Theorem \ref{T:RANK_BOUND}.
\end{proof}

% COROLLARY
In fact, given some mild control on the topology of principal orbits, one can do even better.

\begin{cor}
Let $M^n$ be a smooth, closed $n$-manifold on which a compact Lie group $G$ acts effectively and smoothly.  If the principal $G$-orbits are simply connected and of codimension $d$, then $\rank(G) \leq  \left\lfloor \frac{2(n-d)}{3} \right\rfloor$. 
\end{cor}

% PROOF

\begin{proof}
As the $G$-orbits are homogeneous spaces, they are rationally elliptic.  The maximal torus $T$ of $G$ must act effectively on a principal orbit since, otherwise, the ineffective kernel of the $T$ action would act trivially on the regular part of $M^n$, i.e. on the open, dense collection of all principal $G$-orbits, hence on all of $M^n$, contradicting the effectivity hypothesis for the $G$ action.  As a principal orbit has dimension $n-d$, the result follows.
\end{proof}

% SECTION: PROOF OF THEOREM B - PART (1)     
%-----------------------------------------------------------------
%-----------------------------------------------------------------

\section{Maximal almost-free actions}
\label{S:PROOF_THM_B-(1)}

The existence of an almost-free torus action of maximal rank has strong implications for the topology of the space.  The lemmas in this section together ensure that such a space is rationally homotopy equivalent to one of $\prod \sph^3$, $\sph^2 \x \prod \sph^3$ or $\sph^5 \x \prod \sph^3$, thus verifying Theorem \ref{T:RIGIDITY}\eqref{L:AF}. 

% LEMMA

\begin{lem}
\label{L:d_even}
Let  $M^n$ be a smooth, closed, (simply connected) rationally elliptic $n$-manifold on which  the torus $T^k$ of rank $k = \left\lfloor \frac{n}{3} \right\rfloor$ acts smoothly and almost freely.  Then 
\[
d_2(M) \in \{0,1\} \text{ and } d_{2j}(M) = 0, \text{ for all } j\geq 2,
\] 
where $d_2(M) = 1$ is only possible if $n \equiv 2 \mod 3$.
\end{lem}

% PROOF

\begin{proof}
Observe first that $n = 3k + \mu$, $\mu \in \{0,1,2\}$, and that the long exact homotopy sequence for the homotopy fibration $T^k \to M \to M_T$ yields $d_2(M_T) = d_2(M) + k$ and $d_j (M_T) = d_j(M)$ for all $j \geq 3$.

By Lemma \ref{L:Borel_RE}, $M_T$ is rationally elliptic of formal dimension $n - k$.  Hence, by Proposition \ref{P:ELLIPTIC_INEQ}, 
\[
n - k \geq \sum_{j = 1}^\infty (2j)\, d_{2j} (M_T) = 2k + \sum_{j = 1}^\infty (2j)\, d_{2j} (M)\, ,
\]
from which it follows that
\[
\mu \geq \sum_{j = 1}^\infty (2j)\, d_{2j} (M) \,.
\]
Consequently, if $\mu \in \{0,1\}$, then $d_{2j}(M) = 0$ for all $j \geq 1$, while if $\mu = 2$, then $d_2(M) \in \{0,1\}$ and $d_{2j}(M) = 0$ for all $j \geq 2$.
\end{proof}

This information determines the possibilities for the rest of the rational homotopy groups.

% LEMMA

\begin{lem}
\label{L:RHG}
Let  $M^n$ be a smooth, closed, (simply connected) rationally elliptic $n$-manifold on which  the torus $T^k$ of rank $k = \left\lfloor \frac{n}{3} \right\rfloor$ acts smoothly and almost freely.  Then $n \not\equiv 1 \mod 3$.  Furthermore, if $n \equiv 0 \mod 3$, then 
$$
d_3(M) = k \text{ and } d_j(M) = 0, \text{ for all } j \neq 3,
$$
whereas, if $n \equiv 2 \mod 3$, either
$$
d_3(M) = k-1, \ d_5(M) = 1 \text{ and } d_j(M) = 0, \text{ for all } j \neq 3,5,
$$
or
$$
d_2(M) = 1, \ d_3(M) = k + 1 \text{ and } d_j(M) = 0, \text{ for all } j \neq 2,3.
$$
\end{lem}

% PROOF

\begin{proof}
Since $n = 3k + \mu$, $\mu \in \{0,1,2\}$, and, by Lemma \ref{L:d_even}, for even homotopy groups only $d_2(M)$ is possibly non-trivial, it follows from Proposition \ref{P:ELLIPTIC_INEQ} that
\begin{align*}
- 3\, d_2(M) + 3 \sum_{j = 1}^\infty d_{2j+1}(M) &= - 3 \chi_\pi(M) \\
&\geq 3k \\
&= n - \mu \\
&= - d_2(M) - \mu + \sum_{j = 1}^\infty (2j+1)\, d_{2j+1}(M).
\end{align*}
Hence, 
$$
\mu \geq 2(d_2(M) + d_5(M)) + \sum_{j = 3}^\infty 2(j-1) \, d_{2j+1}(M) \geq 0.
$$
Therefore, $d_2(M) = 0$ and $d_{2j+1}(M) = 0$, for all $j \geq 2$, whenever $\mu \in \{0,1\}$, while for $\mu = 2$ one has $(d_2(M), d_5(M)) \in \{(0,0), (0,1), (1,0)\}$ and $d_{2j+1}(M) = 0$, for all $j \geq 3$.  

By applying the inequality \eqref{E:rat-ell-eq} from Proposition \ref{P:ELLIPTIC_INEQ} once more, the result follows.  Indeed, when $\mu = 0$, one obtains $3k = n = 3 \, d_3(M)$, as desired.  When $\mu = 1$, it is clear that $3k + 1 = n = 3\, d_3(M)$ is impossible.  Finally, when $\mu = 2$, the identity $3k + 2 = n = 3\, d_3(M) + 5\, d_5(M) - d_2(M)$ precludes the case $(d_2(M), d_5(M)) = (0,0)$. 
\end{proof}

It remains to use Lemma \ref{L:RHG} to determine the minimal models, hence rational homotopy types, of $n$-manifolds admitting an almost-free action by a torus of rank $\left\lfloor \frac{n}{3} \right\rfloor$.  The more difficult case, namely, when $d_5(M) = 1$, will be ignored for the moment.

% LEMMA

\begin{lem}
Let $X$ be a (simply connected) rationally elliptic topological space.
\begin{itemize}
\item[(1)]If $d_3(X)=k$ and $d_j(X)=0$ for $j\neq 3$, then $X$ is rationally homotopy equivalent to $\prod^{k}_{i=1} \sph^3$.
\item[(2)]If $d_2(X)=1$, $d_3(X)=k+1$ and $d_j(X)=0$ for $j\neq 2,3$, then $X$ is rationally homotopy equivalent to $\sph^2\x\prod^{k}_{i=1} \sph^3$.
\end{itemize}
\end{lem}

% PROOF

\begin{proof}
In the first case, by the discussion in Subsection \ref{SS:RHT}, the minimal model for $X$ is $(\wedge V_X, d_X)$, where $\wedge V_X = \wedge(x_1, \dots, x_k)$ is the exterior algebra on $k$ elements $x_i$, where $\deg(x_i) = 3$ for all $i = 1, \dots, k$.  Moreover, the differential is trivial, i.e. $d_X (x_i) = 0$ for all $i = 1, \dots, k$, since $\wedge V_X$ has no elements of degree $4$.   Hence, $(\wedge V_X, d_X)$ is precisely the minimal model of $\prod^{k}_{i=1} \sph^3$.

In the second case, the free commutative graded algebra for $X$ is $\wedge V_X = \wedge(u, x_0, \dots, x_k)$, where $\deg(u) = 2$ and $\deg(x_i) = 3$ for all $i = 0, \dots, k$.  Since the differential $d_X$ is decomposable, it follows that $d_X(u) = 0$.  In order to determine $d_X$, the image of 
$$
d_X|_{V^3} : \Span_\Q \{x_0, \dots, x_k\} = \Q^{k+1} \to \Span_\Q\{u^2\} = \Q
$$ 
must be identified.  If the image were trivial, this would imply that, for all $l \in \N$, $H^{2l}(\wedge V_X, d_X) = H^{2l}(X; \Q)$ is non-trivial, contradicting the rational ellipticity assumption.   Because $d_X|_{V^3}$ is linear, it must therefore be surjective.  By a change of basis, it may thus be assumed without loss of generality that $d_X(x_0) = u^2$ and $d_X(x_i) = 0$ for all $i = 1, \dots, k$.  As a consequence,
$$
(\wedge V_X, d_X) = (\wedge(u, x_0), d u = 0, d x_0 = u^2) \ox (\wedge(x_1, \dots, x_k), d x_i =0)
$$
which is the minimal model of $\sph^2 \x \prod^{k}_{i=1} \sph^3$, as desired.
\end{proof}

Now, the case where $n = 3k +2$, $d_3(M) = k-1$, $d_5(M) = 1$ and $d_j(M) = 0$, for $j \neq 3,5$, will follow as a corollary of the general recognition lemma below.

% LEM: RECOGNITION PROD OF ODD DIM SPHERES

\begin{lem}
\label{L:ODD_DIM} Let $X$ be a compact, (simply connected) rationally elliptic topological space such that $d_{2j}(X) = \pi_{2j}(X) \ox \Q=0$, for all $j\geq 1$. If a torus $T^k$ acts almost freely on $X$ and $k=-\chi_\pi(X)$, then $X$ is rationally homotopy equivalent to a product of odd-dimensional spheres.
\end{lem}

% PROOF

% NOTATION

% d_T  - differential in the Borel construction minimal model/
% P_i   - polynomial. 

\begin{proof} Let $(\wedge V_X,d_X)$ be a minimal model for $X$, so that $V_X^{2i}=0$, for all $i\in \N$. Notice that, since $\chi_\pi(X) = -k$, it follows from \cite[Thm.\ 15.11]{FHT} that $\dim_\Q(V_X) = k$.  To prove the lemma, it suffices to show that the differential $d_X$ is the zero map.

By Lemma~\ref{L:Borel_RE}, the Borel construction $X_T$ is rationally elliptic.  The relative minimal model of $X_T$ corresponding to the bundle
$$
X \to X_T \to BT^k
$$
is $(\Q[x_1,\ldots,x_k] \ox \wedge V_X,D)$, where $\deg(x_i) = 2$, for all $i = 1, \dots, k$.  The differential $D$ satisfies $D(x_i)=0$, for all $i = 1, \dots, k$, and $D(v) - d_X(v) \in \Q^+[x_1, \dots, x_k] \ox \wedge V_X$, for all $v \in V_X$.
Thus, it need only be shown that the image of $D|_{V_X}$ lies in $\Q^+[x_1, \dots, x_k] \ox \wedge V_X$, i.e. in the ideal generated by $x_1,\ldots, x_k$. 

Let $\ol V = \Span_\Q \{x_1, \dots, x_k\} \oplus V_X$, so that
\[
\Q[x_1,\ldots,x_k] \ox \wedge V_X = \wedge \ol V.
\]
Note, in particular, that $(\wedge \ol V, D)$ is a minimal model for $X_T$, since $\im(D) \In \wedge^{\geq 2} \ol V$ as a result of $(\wedge V_X, d_X)$ being minimal and all elements of $V_X$ being of degree $\geq 3 > 2 = \deg(x_i)$.

By the minimality of $(\wedge \ol V, D)$, $\dim_\Q(\pi_j(X_T) \ox \Q) = \dim_\Q(\ol V^j)$ (see \cite[Thm.\ 15.11]{FHT}).  Therefore,
\begin{align*}
	\chi_\pi(X_T)	&= \dim_\Q (\ol V^{\mathrm{even}}) - \dim_\Q (\ol V^{\mathrm{odd}}) \\
				&= \dim_\Q (\Span_\Q \{x_1, \dots, x_k\}) - \dim_\Q (V_X) \\
				&= k - k \\
				&= 0.
\end{align*}

It now follows from \cite[Prop.\ 32.10]{FHT} that $(\wedge \ol V, D)$ is a \emph{pure} Sullivan algebra, i.e.~there is a differential-preserving isomorphism
$$
\Phi : (\wedge \ol V, D) \to (\wedge(U \oplus W),d),
$$
where $U = U^{\mathrm{odd}}$, $W = W^{\mathrm{even}}$, $d(W) = \{0\}$ and $d(U) \In \wedge W$.  The isomorphism $\Phi$ induces a linear isomorphism
$$
\vphi : \ol V \to U \oplus W
$$
of graded vector spaces, such that, for every $\ol v \in \ol V$, 
$$
\Phi(\ol v) - \vphi(\ol v) \in \wedge^{\geq 2}(U \oplus W).
$$

The proof that $D(V_X) \In \Q^+[x_1, \dots, x_k] \ox \wedge V_X$ will be done by induction on degree.
First, since there are no non-trivial elements of degree $< 4$ in $\wedge^{\geq 2} \ol V$, it follows that $\Phi(\ol v) = \vphi(\ol v)$ whenever $\ol v \in \ol V$ with $\deg(\ol v) \leq 3$.  Therefore, the maps
\begin{align*}
\Phi|_{\ol V^2} = \vphi|_{\ol V^2} &: \ol V^2 = \Span_\Q \{x_1, \dots, x_k\} \to W   \\
\text{ and } \ \Phi|_{\ol V^3} = \vphi|_{\ol V^3} &: \ol V^3 = V_X^3 \to U^3
\end{align*}
are isomorphisms.  Hence, for any $v \in V_X^3 = \ol V^3$, one has $\Phi(v) \in U^3$ and, consequently, $\Phi(D(v)) = d(\Phi(v)) \in \wedge W = \Phi(\Q[x_1, \dots, x_k])$.  As $\Phi$ is injective, this implies that $D(v) \in \Q^+[x_1, \dots, x_k] \In \Q^+[x_1, \dots, x_k] \ox \wedge V_X$, as desired.

Suppose now that $D(V_X^{\leq 2j-1}) \In \Q^+[x_1, \dots, x_k] \ox \wedge V_X$.  Let $v \in V_X^{2j+1}$.  Then there is some $y \in \wedge^{\geq 2} (U \oplus W)$ such that $\Phi(v) = \vphi(v) + y$.  Since $\Phi$ is surjective, there is a $\ol y \in \wedge^{\geq 2} \ol V$ such that $\Phi(\ol y) = y$.  Therefore, $\Phi(v - \ol y) = \vphi(v) \in U$ and, as a result, 
$$
\Phi(D(v-\ol y)) = d(\Phi(v - \ol y)) \in d(U) \In \wedge W= \Phi(\Q[x_1, \dots, x_k]).
$$
By the injectivity of $\Phi$, this implies that $D(v) - D(\ol y) \in \Q[x_1, \dots, x_k]$.  However, since $\ol y \in \wedge^{\leq 2} \ol V$ is a linear combination of products of elements of degree $\leq 2j-1$, the induction hypothesis ensures that $D(v) \in \Q^+[x_1, \dots, x_k] \ox \wedge V_X$.

Hence, by induction, $\im(D|_{V_X}) \In \Q^+[x_1, \dots, x_k] \ox \wedge V_X$, as desired.
\end{proof}

%%%%%%%%%%%%%%%%%%%%%%%%%%%%%%%%%%%

% COROLLARY

\begin{cor}
Let $M^n$, $n = 3k + 2$, be a smooth, closed, (simply connected) rationally elliptic, $n$-dimensional manifold on which the torus $T^k$ acts almost freely.  Suppose further that $d_3(X) = k-1$, $d_5(X) = 1$ and $d_j(X) = 0$ for all $j \neq 3, 5$.  Then $M^n$ is rationally homotopy equivalent to $\sph^5 \x \prod^{k-1}_{i=1} \sph^3$.
\end{cor}

% PROOF

\begin{proof}
The rational homotopy type follows immediately from Lemma \ref{L:ODD_DIM}.
\end{proof}

% REMARK

\begin{rem}
Observe that the smoothness assumptions played no role in the arguments in this section.  Therefore, the classification obtained in Theorem \ref{T:RIGIDITY}\eqref{L:AF} holds, more generally, for compact, rationally elliptic topological spaces of formal dimension $n$ which admit a maximal-rank, almost-free torus action. 

The question of whether one gets a classification up to equivariant rational homotopy equivalence in the case of almost-free torus actions of maximal rank is still open.  The main obstacle seems to be the abundance of maximal-rank, almost-free torus actions on $\prod S^3$.
\end{rem}

%%%%%%%%%%%%%%%%%%%%%%%%%%%%%%%%%%%%%%%%%%%%%%%%%

% SECTION: PROOF OF THM B - PART (2)
%---------------------------------------------------------
%---------------------------------------------------------

\section{Maximal effective actions}
\label{S:MaxEff}

It turns out that effective torus actions of maximal rank are special cases of a more general type of action, namely slice-maximal actions, as defined in \cite{GKRW} (see also \cite{Is, Us}):  A smooth, effective action of the torus $T^k$ on a smooth, closed $n$-manifold $M^n$ is called \emph{slice maximal} if $n=k+s$, where $s$ is the maximal dimension of an isotropy subgroup.  

% LEMMA

\begin{lem}
\label{L:SliceMax}
Let $M^n$ be a smooth, closed, (simply connected) rationally elliptic, $n$-dimensional manifold which admits a smooth, effective action of the torus $T^k$ of rank $k=\lfloor\frac{2n}{3}\rfloor$.  Then the $T^k$ action is slice maximal.

Moreover, if $n \not\equiv 1 \mod 3$, there is a rank-$\lfloor\frac{n}{3}\rfloor$ subtorus of $T^k$ acting almost freely on $M^n$, while if $n \equiv 1 \mod 3$, there is an almost-free action by a subtorus of rank $\lfloor\frac{n}{3}\rfloor - 1$.
\end{lem}

% PROOF

\begin{proof}
Let $s > 0$ be the maximal dimension of an isotropy subgroup of the $T^k$ action and let $p \in M^n$ be such that the isotropy subgroup $T_p$ at $p$ has dimension $s$.  It is known from the arguments in Section \ref{S:PROOF_THM_A} used to prove Theorem \ref{T:RANK_BOUND} that $k + s \leq n$ and that there is a subtorus of rank $k-s$ acting almost freely on $M^n$, hence $3(k-s) \leq n$.  By hypothesis, there is some $a \in \{0,1,2\}$ such that $2n = 3k + a$.  

Suppose that $n > k+s$.  Then $a \in \{1,2\}$, since
$$
n \geq 3(k-s) > 3k - 3(n-k) = 6k - 3n
$$ 
implies $2n > 3k$.  Now, from $3(k-s) \leq n$, one observes that $6s \geq 6k - 2n = 3k - a$, which in turn yields $2s \geq k$, since $6s$ is divisible by $3$ and $a \in \{1,2\}$.

On the other hand,
$$
2s < 2(n-k) = (3k + a) - 2k = k + a,
$$
from which one concludes that $k \leq 2s < k + a$.  

If $a=1$, then $k = 2s$ and, hence, $2n = 6s +1$, which is impossible.  If $a = 2$, then $k$ is even, as $2n = 3k + 2$.  Therefore $k = 2s$, $n = 3s + 1$ and $k-s = s = \lfloor\frac{n}{3}\rfloor$, which contradicts Lemma \ref{L:RHG}, i.e. if $n \equiv 1 \mod 3$, then $M^n$ cannot admit an almost-free action of rank $\lfloor\frac{n}{3}\rfloor$.  It thus follows that $n = k + s$, hence that the $T^k$ action is slice maximal, as desired.

The identities $n = k+s$ and $2n = 3k + a$ yield $k = 2s - a$, hence $n = 3s - a$ and $k-s = s - a$.  By considering each $a \in \{0,1,2\}$ in turn, the remaining statements follow easily.
\end{proof}

In \cite{GKRW} rationally elliptic manifolds admitting slice-maximal torus actions have been classified up to equivariant rational homotopy equivalence, which allows the proof of Theorem \ref{T:RIGIDITY} to be completed.  Indeed, it was shown that if $M^n$ admits a slice-maximal $T^k$ action, it must then be ($T^k$-equivariantly) rationally homotopy equivalent to the quotient $M'$ of a product of spheres $\prod_i \sph^{n_i}$, $n_i \geq 3$, by a free, linear $T^l$ action.  The long exact sequence of homotopy groups for the principal bundle $T^l \to \prod_i \sph^{n_i} \to M'$ yields $d_2(M) = l$ and $d_j(M) = d_j (\prod_i \sph^{n_i})$, for all $j \geq 3$. Because $d_j(\sph^k)$ is nonzero (in fact, equal to $1$) only for $j=k$ and, when $k$ is even, for $j=2k-1$, the numbers $d_j(M)$ completely determine the dimensions of the spherical factors in $\prod_i \sph^{n_i}$.

% THM

\begin{thm}
\label{T:ThmB2}
Let $M^n$, $n \geq 3$, be an $n$-dimensional, smooth, closed, (simply connected) rationally elliptic manifold equipped with a smooth, effective action of the torus $T^k$ of rank $\left\lfloor\frac{2n}{3}\right\rfloor$.  Then $M^n$ is $T^k$-equivariantly rationally homotopy equivalent to a manifold of one of the following forms:
\begin{enumerate}
\item $X \x \prod \sph^3$,  with $X \in \{\sph^3, \sph^4, \sph^5, \sph^7, \sph^5 \x \sph^5\}$; \vspace*{1mm}
\item $(Y \x \prod \sph^3)/\sph^1$,  with $Y \in \{\sph^3, \sph^5\}$; or \vspace*{1mm}
\item $(\prod \sph^3)/T^2$.
\end{enumerate}
\end{thm}

% PROOF

\begin{proof}
 When $n \not\equiv 1 \mod 3$, the possible rational homotopy types are given by Theorem \ref{T:RIGIDITY}\eqref{L:AF}, established in Section \ref{S:PROOF_THM_B-(1)}, due to  the existence of an almost-free action by a subtorus of rank $\left\lfloor\frac{n}{3}\right\rfloor$.  Note, in particular, that $\sph^2 \x \prod \sph^3 \simeq_\Q (\prod \sph^3)/\sph^1$ for every free, linear $\sph^1$ action on $\prod \sph^3$.

Suppose now that $n \equiv 1 \mod 3$.  By the discussion above, in order to determine the possible rational homotopy types, it suffices to determine the possible dimensions $d_j(M)$ of all rational homotopy groups.

Let $n = 3l+1$, $l \geq 1$, and let $s>0$ be the maximal dimension of an isotropy subgroup.  Then $k = 2l$ and, by Lemma \ref{L:SliceMax}, $k-s = l-1$.  Hence $l = s-1$, and $n$ and $k$ can be rewritten as $n = 3(k-s) + 4 = 3s - 2$ and $k = 2(s-1)$, respectively.  By repeating the analysis in the proof of Lemma \ref{L:d_even} (with $\mu = 4$ and $k$ replaced by $k-s$), one obtains 
\[
4 \geq \sum_{j = 1}^\infty (2j)\, d_{2j} (M) ,
\]
from which it immediately follows that 
\[
(d_2(M), d_4(M)) \in \{(0,0), (1,0), (2,0),(0,1)\}
\] 
and $d_{2j} = 0$, for all $j \geq 3$.  Similarly, by repeating the arguments in the proof of Lemma \ref{L:RHG}, one obtains
\[
4 \geq 2\, d_2(M) + \sum_{j = 2}^\infty 2(j-1) \, d_{2j+1}(M),
\]
hence $d_{2j+1}(M) = 0$, for all $j \geq 4$, and 
\[
4 \geq 2(d_2(M) + d_5(M)) + 4\, d_7(M).
\]
This inequality, together with the identity $n = 3\, d_3(M) + 5\, d_5(M) + 7\, d_7(M) - d_2(M) - 3 \, d_4(M)$ from Proposition \ref{P:ELLIPTIC_INEQ}, yields that the only possibilities are
\[
(d_2(M), d_4(M), d_5(M), d_7(M)) \in \left\{
\begin{array}{ccc}
(1,0,1,0), & (0,0,2,0), & (0,1,2,0) \\
(2,0,0,0),& (0,0,0,1), & (0,1,0,1)
\end{array}
\right\} .
\]
Observe that $(d_2(M), d_4(M), d_5(M), d_7(M)) = (0,1,2,0)$ cannot occur, since $d_4(M) \neq 0$ requires $d_7(M) \neq 0$.  Finally, in each remaining case it is easy to determine $d_3(M)$ and, consequently, $M^n$ is rationally homotopy equivalent to one of the following manifolds:
\[
\begin{array}{c|c}
 & \\[-2mm]
(d_2(M), d_4(M), d_5(M), d_7(M)) & M^n \simeq_\Q \\[2mm] \hline
& \\[-3mm]
(0,1,0,1) & \sph^4 \x \prod_{i=1}^{s-2} \sph^3 \\[2mm] 
(0,0,0,1) &  \sph^7 \x \prod_{i=1}^{s-3} \sph^3 \\[2mm] 
(0,0,2,0) &  \sph^5 \x \sph^5 \x \prod_{i=1}^{s-4} \sph^3 \\[2mm]
(1,0,1,0) &  (\sph^5 \x \prod_{i=1}^{s-2} \sph^3)/\sph^1 \\[2mm]
(2,0,0,0) &  (\prod_{i=1}^{s} \sph^3)/T^2\\[2mm] 
\end{array}
\]

The $T^k$-equivariance comes directly from \cite{GKRW}.
\end{proof}

It remains only to show that the manifolds arising in Theorem \ref{T:ThmB2} fall into only finitely many rational homotopy types.  The more difficult case of $(\prod_{i=1}^{s} \sph^3)/T^2$ will be postponed until Section \ref{S:Biquotients}.

% PROPOSITION

\begin{prop}
Suppose $\sph^1$ acts freely and linearly on $\sph^5 \x \prod_{i=1}^{m} \sph^3$.  Then the quotient $(\sph^5 \x \prod_{i=1}^{m} \sph^3)/\sph^1$ is rationally homotopy equivalent to either $\CP^2\x \prod_{i=1}^{m} \sph^3$ or $\sph^2 \x \sph^5 \x \prod_{i=1}^{m-1} \sph^3$.
\end{prop}

% PROOF

\begin{proof}
For the sake of notation, let $P = \sph^5 \x \prod_{i=1}^{m} \sph^3$.  First note that, since $\sph^1$ acts freely on $P$, there is a principal $\sph^1$-bundle $\sph^1 \to P \to P/\sph^1$.  As $\sph^1$ also acts (freely) on the contractible space $E\sph^1$, there is an associated bundle $E\sph^1 \to P_{\sph^1} \to P/\sph^1$, where $P_{\sph^1}$ is the Borel construction.  Hence, $P_{\sph^1}$ and $P/\sph^1$ are homotopy equivalent, and the fibre bundle $P \to P_{\sph^1} \to B\sph^1$ associated to the universal $\sph^1$-bundle becomes (up to homotopy) 
\[
P \to P/\sph^1 \to B\sph^1 \,.
\]

The minimal models of $P$ and $B\sph^1$ are given by $(\wedge(x_1, \dots, x_m, y), 0)$ and (the polynomial algebra) $(\Q[u],0)$ respectively, where $\deg(x_i) = 3$, for all $i = 1, \dots, m$, $\deg(y) = 5$ and $\deg(u) = 2$.  Then the relative minimal model for $P/\sph^1$ is given by 
\[
(\Q[u] \ox \wedge(x_1, \dots, x_m, y), D)
\]
with $D(u) = 0$, $D(x_i) = \lambda_i  u^2 \in \Span_\Q\{u^2\}$, $i = 1, \dots, m$, and $D(y) = \alpha u^3 \in \Span_\Q \{u^3\}$.

Suppose first, some $\lambda_i$ is nonzero. Without loss of generality, $\lambda_1\neq 0$.  A change of basis via $\ol x_1 = \frac{1}{\lambda_1} x_1$, $\ol x_i = x_i - \lambda_i x_1$, $i = 2, \dots, m$, and $\ol y = y - \alpha \ol x_1 u$, therefore yields 
\[
D(\ol x_1) = u^2, \ D(\ol x_i) = 0, \ i = 2, \dots, m, \ \text{ and } \ D(\ol y) = 0. 
\]
The relative minimal model $(\Q[u] \ox \wedge(x_1, \dots, x_m, y), D)$ is then, in fact, a minimal model, namely that of $\sph^2 \x \sph^5 \x \prod_{i=1}^{m-1} \sph^3$.

Suppose now that $D(x_i) = 0$, for all $i = 1, \dots, m$.  Then $D(y) = \alpha u^3 \neq 0$, since otherwise the manifold $P/\sph^1$ would have infinite formal dimension.  Setting $\ol y = \frac{1}{\alpha}y$ yields $D(\ol y) = u^3$, and the relative minimal model $(\Q[u] \ox \wedge(x_1, \dots, x_m, y), D)$ is then the minimal model of $\CP^2 \x \prod_{i=1}^m \sph^3$.
\end{proof}

% REMARK

\begin{rem}
The fact that, in each dimension, there are only finitely many rational homotopy types of manifolds $(\sph^5 \x \prod_{i=1}^{m} \sph^3)/\sph^1$ and $(\prod_{i=1}^{m} \sph^3)/T^2$ is in stark contrast to the situation for ordinary homotopy types.  Indeed, in \cite{DV2, CE, Kr} it has been shown that, already in dimension $7$, there are infinitely many distinct homotopy types of such manifolds, distinguished by their cohomology rings.
\end{rem}

In the proof of Theorem \ref{T:ThmB2}, the only case where the existence of an effective torus action of maximal rank is truly required is when 
\[
(d_2(M), d_4(M), d_5(M), d_7(M)) = (2,0,0,0).
\]  
In all other cases, in order to compute the minimal model, it suffices to know that there is an almost-free torus action of rank $\left\lfloor\frac{n}{3}\right\rfloor$ (for $n \not\equiv 1 \mod 3$) or $\left\lfloor\frac{n}{3}\right\rfloor - 1$ (for $n \equiv 1 \mod 3$).  If, in the exceptional case, one assumes only the existence of an almost-free torus action of rank $\left\lfloor\frac{n}{3}\right\rfloor - 1$, then the result becomes much less rigid.

% PROP

\begin{prop}
\label{P:AMAF}
In each dimension $n = 3 m + 4 \not\equiv 0 \mod 4$, there are infinitely many rational homotopy types of closed, smooth, (simply connected) rationally elliptic manifolds which admit a free torus action of rank $\left\lfloor\frac{n}{3}\right\rfloor - 1 = m$, but which do not admit an effective torus action of rank $\left\lfloor\frac{2n}{3}\right\rfloor$.
\end{prop}

% PROOF

\begin{proof}
Fix a dimension $n = 3m+4\not\equiv 0 \mod 4$.  For each $\alpha \in \Z \backslash \{0\}$, consider the minimal model $(\wedge V, d_\alpha)$, where 
$$
\wedge V = \wedge(u_1, u_2, x_1, \dots, x_{m+2}),
$$ 
with $\deg(u_i) = 2$, $i = 1,2$, $\deg(x_j) = 3$, $j = 1, \dots, m+2$, and the differential is given by $d_\alpha(u_i) = 0$, $d_\alpha(x_1) = u_1 u_2$, $d_\alpha(x_2) = u_1^2 + \alpha u_2^2$ and $d_\alpha(x_j) = 0$, for all $j = 3, \dots, m+2$.  It is easy to verify that two such models, $(\wedge V, d_\alpha)$ and $(\wedge V, d_\beta)$, are isomorphic if and only if there is some $c \in \Q$ such that $\beta = c^2 \alpha$.

Since $n \not\equiv 0 \mod 4$, by \cite[Thm.\ 3.2]{FOT}, there is a smooth, closed, (simply connected) rationally elliptic manifold $M^n_\alpha$ with minimal model $(\wedge V, d_\alpha)$.  Recall that the minimal model of $BT^m$ is $(\Q[v_1, \dots, v_m], 0)$, with $\deg(v_l) = 2$, for all $l = 1, \dots, m$.  Define a relative minimal model
$$
(\Q[v_1, \dots, v_m], 0) \to (\Q[v_1, \dots, v_m] \ox \wedge V, D_\alpha) \to (\wedge V, d_\alpha),
$$
where $D_\alpha(v_l) = 0$, for all $l = 1, \dots, m$, $D_\alpha(x_1) = d_\alpha (x_1)$, $D_\alpha(x_2) = d_\alpha (x_2)$ and $D_\alpha(x_j) = v_{j-2}^2$, for $j = 3 \dots m$.  

Then $(\Q[v_1, \dots, v_m] \ox \wedge V, D_\alpha)$ is, in fact, a minimal model and 
$$
\dim_\Q H^*(\Q[v_1, \dots, v_m] \ox \wedge V, D_\alpha) < \infty.
$$
As this model has formal dimension $n-m = 2m + 4 \not\equiv 0 \mod 4$, \cite[Thm.\ 3.2]{FOT} again implies that there is a smooth, closed, simply connected, $(n-m)$-dimensional manifold $N_\alpha$ with minimal model $(\Q[v_1, \dots, v_m] \ox \wedge V, D_\alpha)$.

Now, by \cite[Prop.\ 7.17]{FOT} (see also \cite[Prop.\ 4.2]{Ha} and \cite[Prop.\ 4.3.20]{AP}), there is a smooth, closed, simply connected $n$-manifold $M'_\alpha$, with the same rational homotopy type as $M_\alpha$, on which the torus $T^m$ acts freely with quotient $N_\alpha$.

Finally, by Theorem \ref{T:ThmB2}, if $M'_\alpha$ admits an effective action by a torus of rank $\left\lfloor\frac{2n}{3}\right\rfloor$, it must be rationally homotopy equivalent to a manifold of the form $(\prod_{i=1}^{m+2} \sph^3)/T^2$.  However, it will be shown in Theorem \ref{T:biqs} that such a manifold has a minimal model of the form $(\wedge V, d_\alpha)$ if and only if $\alpha = \pm 1$. 
\end{proof}

% SECTION: BIQUOTIENTS
%-------------------------------------
%-------------------------------------

\section{Quotients of free, linear $T^2$ actions on $\prod \sph^3$}
\label{S:Biquotients}

In this section, it is shown that, for each $N \in \N$, there are only finitely many rational homotopy types of manifolds given by quotients of $\prod_{i = 1}^N \sph^3$ by a free, linear $T^2$ action. 
Recall first that, up to equivariant diffeomorphism, there is a unique (smooth) effective $T^2$ action on $\sph^3$, given by
\[
(z,w) \cdot q = z u + w v j,
\]
where $z,w \in \sph^1 \in \C$ and $q = u + vj \in \sph^3 \In \HH$, for $u, v \in \C$ with $|q| = |u|^2 + |v|^2 = 1$.  As a consequence, any linear, effective $T^2$ action on a product $\prod_{i = 1}^N \sph^3$ arises from a homomorphism $T^2 \to T^{2N}$ and can be written in the form
\beq
\label{Eq:action}
(z,w) \cdot \underline q = 
\bpm 
z^{a_1} w^{k_1} u_1 + z^{b_1} w^{l_1} v_1 j \\ 
\vdots \\ 
z^{a_N} w^{k_N} u_N + z^{b_N} w^{l_N} v_N j 
\epm,
\eeq
where $\underline q = (q_1, \dots, q_N)^t \in \prod_{i = 1}^N \sph^3$, with $q_i = u_i + v_i j \in \sph^3$ as above, and the integers $a_i$, $b_i$, $k_i$ and $l_i$ satisfy $\gcd(a_1, \dots, a_N, b_1, \dots, b_N) = 1$ and  $\gcd(k_1, \dots, k_N, l_1, \dots, l_N) = 1$ (to ensure effectiveness).

It is a simple exercise to check that such an action is free if and only if, for all choices $(c_i, m_i) \in \{(a_i, k_i), (b_i, l_i)\}$, one has
\beq
\label{Eq:freeness}
\gcd\left\{\bvm c_i & c_j \\ m_i & m_j \evm \Bigm| 1 \leq i < j \leq N \right\} = 1,
\eeq 
where, for any matrix $A$, $|A|$ denotes its determinant.

% THM: BIQUOTIENTS

\begin{thm}
\label{T:biqs}
Suppose that a manifold $M$ arises as the quotient of $\prod_{i = 1}^N \sph^3$, $N \geq 3$, by a free, linear $T^2$ action.  Then $M$ is rationally homotopy equivalent to either 
\begin{align*}
(\sph^2 \x \sph^2) &\x \prod_{i = 1}^{N-2} \sph^3, \\
(\CP^2 \# \CP^2) &\x \prod_{i = 1}^{N-2} \sph^3, \\
\textrm{or } \  T^1 (\sph^2 \x \sph^2) &\x \prod_{i = 1}^{N-3} \sph^3,
\end{align*}
where $T^1 (\sph^2 \x \sph^2) $ denotes the unit tangent bundle of $\sph^2 \x \sph^2$.
\end{thm}

In order to establish Theorem \ref{T:biqs}, the following lemma will be useful.

% LEMMA

\begin{lem}
\label{L:reduced}
Suppose that $T^2$ acts freely and linearly on $\prod_{i = 1}^N \sph^3$ via an action of the form \eqref{Eq:action}.  Then it may be assumed, without loss of generality, that $a_1 \neq 0$, $k_1 = 0$, $(b_1, l_1) \neq (0,0)$ and $k_2 l_2 \neq 0$.
\end{lem}

% PROOF

\begin{proof}
Suppose first that $a_ib_i = 0$ for all $i = 1, \dots, N$.  For each $i$, set $c_i$ to be whichever of $a_i$ and $b_i$ is equal to zero.  However, by the freeness condition \eqref{Eq:freeness}, this is impossible.  Indeed, it would imply that there is some point with isotropy group containing an $\sph^1$.  Thus there is some $i \in \{1, \dots, N\}$ such that $a_i b_i \neq 0$.  As swapping factors in $\prod_{i = 1}^N \sph^3$ is an equivariant diffeomorphism, it may be assumed that $i = 1$.

Consider now the term $z^{a_1} w^{k_1}$ in the first factor.  If $d = \gcd(a_1, k_1) \neq 0$, set $m = a_1/d$ and $n = k_1/d$.  In particular, there are integers $r,s \in \Z$ satisfying $ms - nr = 1$.  The entire action of $T^2$ can be reparametrised by $x = z^m w^n$ and $y = z^r w^s$, while ensuring that effectiveness is maintained.  In this new parametrisation, the old term $z^{a_1} w^{k_1}$ becomes $x^d$.  

Similarly, the old term $z^{b_1} w^{l_1}$ becomes $x^{b_1 s - l_1 r} y^{-b_1 n + l_1 m}$.  As $ms - nr = 1$ and $b_1 \neq 0$, these indices cannot be simultaneously zero.  Thus, after relabelling $x,y$ with $z,w$ and relabelling the indices in the new parametrisation appropriately, it may be assumed without loss of generality that the indices of the action on the first factor satisfy $a_1 \neq 0$, $k_1 = 0$ and $(b_1, l_1) \neq (0,0)$.

Given now $k_1 = 0$, it follows from freeness, by the same argument as for $a_i b_i$ above, that there must be some $i > 1$ such that $k_i l_i \neq 0$.  By swapping factors if necessary, it may be assumed without loss of generality that $i = 2$.
\end{proof}

The following technical lemma will be crucial in the proof of Theorem~\ref{T:biqs}.

% LEMMA: TECHNICAL LEMMA

\begin{lem}
\label{L:same}
Suppose that $a_i, b_i, k_i, l_i \in \Z$, $i = 1, \dots, N$, are integers for which the conditions in \eqref{Eq:freeness} hold and such that $a_1 \neq 0$, $k_1 = 0$, $l_1 \neq 0$ and $k_2 l_2 \neq 0$.  Suppose further that $\gcd(b_1, l_1) = 1$.  Then the matrix
\beq
\label{Eq:Matrix}
\bpm
b_1 & a_2 b_2 & \dots & a_N b_N\\
l_1 & a_2 l_2 + b_2 k_2 & \dots & a_N l_N + b_N k_N\\
0 & k_2 l_2 & \dots & k_N l_N
\epm
\eeq
has rank $\geq 2$.  If the rank is precisely $2$ then there exists $\varepsilon \in \{ \pm 1\}$ such that, for all $j = 2, \dots, N$,
\[
\bvm b_1 & a_j \\ l_1 & k_j \evm  \bvm b_1 & b_j \\ l_1 & l_j \evm = \varepsilon k_j l_j.
\]
\end{lem}

% PROOF

\begin{proof}
First notice that the statement is trivial for $N = 2$, since the terms on the left- and right-hand side must each be equal to $\pm 1$ by considering the conditions \eqref{Eq:freeness}.  Here it is important that $a_1 \neq 0$.

From now on assume that $N \geq 3$.  The rank of the matrix is clearly at least two, since the first two columns are linearly independent.  If the rank is precisely $2$ then, for all $i = 3, \dots, N$, there exist $\lambda_i, \mu_i \in \Q$ such that
\begin{align}
\label{Eq:LD1}
a_i b_i &= \lambda_i b_1 + \mu_i a_2 b_2 \\
\label{Eq:LD2}
a_i l_i + b_i k_i &= \lambda_i l_1 + \mu_i (a_2 l_2 + b_2 k_2) \\
\label{Eq:LD3}
k_i l_i &= \mu_i k_2  l_2.
\end{align}
For all $j = 2, \dots, N$, define
\[
x_j = \bvm b_1 & a_j \\ l_1 & k_j \evm \bvm b_1 & b_j \\ l_1 & l_j \evm \ 
\textrm{ and } \  
y_j = k_j l_j.
\]
By \eqref{Eq:LD3}, $y_i = \mu_i y_2$, for all $i = 3, \dots, N$.  On the other hand, from \eqref{Eq:LD1}, \eqref{Eq:LD2} and \eqref{Eq:LD3} it follows that, for all $i = 3, \dots, N$,
\begin{align*}
x_i &= b_1^2 k_i l_i - b_1 l_1(a_i l_i + b_i k_i) + l_1^2 a_i b_i \\
&= \mu_i b_1^2 k_2 l_2 - b_1 l_1 (\lambda_i l_1 + \mu_i (a_2 l_2 + b_2 k_2)) + l_1^2 (\lambda_i b_1 + \mu_i a_2 b_2) \\
&= \mu_i x_2 - \lambda_i b_1 l_1^2 + \lambda_i b_1 l_1^2 \\
&= \mu_i x_2.
\end{align*}
Therefore, since $y_2 \neq 0$, the matrix 
\[
\bpm
x_2 & x_3 & \dots & x_N \\
y_2 & y_3 & \dots & y_N
\epm
=
\bpm
x_2 & \mu_3 x_2 & \dots & \mu_N x_2 \\
y_2 & \mu_3 y_2 & \dots & \mu_N y_2
\epm
\]
has rank 1 and the rows must be linearly dependent.  Thus there are integers $r, s \in \Z$ with $\gcd(r,s) = 1$ such that 
\[
r x_j = s y_j \ \textrm{ for all } \ j = 2, \dots, N.  
\]

It turns out that $s=\pm 1$. Indeed, otherwise $s = 0 \mod p$, for some prime $p > 1$.  Since $\gcd(r,s) = 1$, it would then follow that $x_j = 0 \mod p$, for all $ j = 2, \dots, N$.  Hence, for each $j = 2, \dots, N$, one could choose $(c_j, m_j) \in \{(a_j, k_j), (b_j, l_j)\}$ such that $\left| \begin{smallmatrix} b_1 & c_j \\ l_1 & m_j \end{smallmatrix} \right| = 0 \mod p$.

By the linearity of the determinant in the second column, for every $2 \leq j_1 < j_2 \leq N$ one has (modulo $p$)
\begin{align*}
0=- m_{j_2} \bvm b_1 & c_{j_1} \\ l_1 & m_{j_1} \evm + m_{j_1} \bvm b_1 & c_{j_2} \\ l_1 & m_{j_2} \evm &= 
 l_1 \bvm c_{j_1} & c_{j_2} \\ m_{j_1} & m_{j_2} \evm \\
\intertext{as well as}
0=-c_{j_2} \bvm b_1 & c_{j_1} \\ l_1 & m_{j_1} \evm + c_{j_1} \bvm b_1 & c_{j_2} \\ l_1 & m_{j_2} \evm
&= b_1 \bvm c_{j_1} & c_{j_2} \\ m_{j_1} & m_{j_2} \evm.
\end{align*}
Since $\gcd(b_1, l_1) = 1$, it would follow that $\left| \begin{smallmatrix} c_{j_1} & c_{j_2} \\ m_{j_1} & m_{j_2} \end{smallmatrix} \right| = 0 \mod p$, for every $2 \leq j_1 < j_2 \leq N$.  However, this would ensure the existence of pairs $(c_1, m_1), \dots, (c_N, m_N)$ such that the condition \eqref{Eq:freeness} fails, contradicting the hypothesis.

As a consequence, $r \neq 0$ as, otherwise, $y_2 = 0$, which contradicts the hypothesis $k_2 l_2 \neq 0$.  
Moreover, any prime divisor of $r$ divides $y_j$, hence either $k_j$ or $l_j$, for all $j = 2, \dots, N$.  By setting $(c_1, m_1) = (a_1, k_1) = (a_1, 0)$ and by choosing appropriate $(c_j, m_j)$, $j = 2, \dots, N$, one readily finds a contradiction to the hypothesis that \eqref{Eq:freeness} holds.  As $r \neq 0$, it follows that $r = \pm 1$. This completes the proof.
\end{proof}

As illustrated in the lemma below, it is often possible to reduce minimal models to a simpler form.

% LEMMA

\begin{lem}
\label{L:easy}
Suppose that $(\Q[s_1, s_2] \ox \wedge(x_1, \dots, x_N), D)$, with $\deg(s_1)= \deg(s_2) = 2$ and $\deg(x_i) = 3$ for all $i = 1, \dots, N$, is a minimal model whose differential satisfies either
\begin{align*}
D(x_1) &= \alpha s_1^2, \\
D(x_2) &= \beta s_1 s_2 + \gamma s_2^2, \\
\intertext{where $\alpha, \gamma \neq 0$, or}
D(x_1) &= s_1 s_2, \\
D(x_2) &= s_1^2 + s_2^2.\\
\intertext{Then $(\Q[s_1, s_2] \ox \wedge(x_1, \dots, x_N), D)$ can be rewritten in the form $(\Q[\tilde s_1, \tilde s_2] \ox \wedge(\tilde x_1, \tilde x_2, x_3 \dots, x_N), D)$ such that $D$ satisfies}
D(\tilde x_1) &= \tilde s_1^2, \\
D(\tilde x_2) &= \tilde s_2^2.
\end{align*}
\end{lem}

% PROOF

\begin{proof}
In the first case, if $\beta = 0$ the statement is trivially true by rescaling $x_1$ and $x_2$.  Suppose $\beta \neq 0$.  The desired change of basis is then given by 
\[
\tilde s_1 = \frac{\beta}{2\gamma} s_1, \ \ \tilde s_2 = \tilde s_1 + s_2, \ \ \tilde x_1 = \frac{\beta^2}{4\alpha \gamma^2} x_1 \textrm{ and } \tilde x_2 = \tilde x_1 + \frac{1}{\gamma} x_2.
\]
In the second case, the appropriate change is given by
\[
\tilde s_1 = s_1 - s_2, \ \ \tilde s_2 =  s_1 + s_2, \ \ \tilde x_1 = x_2 - 2 x_1 \textrm{ and } \tilde x_2 = x_2 + 2 x_1.
\]

\end{proof}

% PROOF OF THEOREM ON BIQUOTIENTS

\begin{proof}[Proof of Theorem \ref{T:biqs}]
Following the discussion before the statement of the theorem, every free, linear $T^2$ action on $\prod_{i=1}^N \sph^3$ is equivariantly diffeomorphic to one of the form \eqref{Eq:action}.  As a consequence, only such actions need be considered.  Moreover, every such action is, in fact, a biquotient action.  That is, there is a homomorphism $f : T^2 \to \prod \sph^3 \x \prod \sph^3$ yielding a free two-sided action of $T^2$ on the Lie group $\prod \sph^3$.  On the $i^{\rm th}$ factor this action is given by
\[
(z,w) \cdot q_i = z^{a_i} w^{k_i} u_i + z^{b_i} w^{l_i} v_i j = \left(z^{\tfrac{a_i + b_i}{2}} w^{\tfrac{k_i + l_i}{2}}\right) q_i \left(\bar z^{\tfrac{b_i - a_i}{2}} \bar w^{\tfrac{l_i - k_i}{2}}\right).
\]
Since the parity of $a_i \pm b_i$ (resp.\ $k_i \pm l_i$) does not depend on the choice of sign, the action is well defined.

Recall that a Lie group $L$ has the rational homotopy type of a product $\sph^{2m_1 - 1} \x  \dots \x \sph^{2m_{r} - 1}$ of odd-dimensional spheres, with $r = \rank(L)$, and its minimal model is hence given by $(H^*(L; \Q), d) = (\wedge (x_1, \dots, x_{r}), 0)$, where $\deg(x_i) = 2m_i - 1$, for $i = 1, \dots, r$.  It is then easy to see that the classifying space $BL$ has minimal model $(H^*(BL;\Q), \bar d) = (\Q[\bar x_1, \dots, \bar x_{r}], \bar d)$, where  the $\bar x_i$ are the transgressions of the $x_i$ in the Serre spectral sequence for the universal bundle $L \to EL \to BL$ and satisfy $\deg(\bar x_i) = 2m_i$ and $\bar d(\bar x_i) = 0$ for all $i = 1, \dots, r$.  Then the minimal model of a biquotient $G \bq H$, computed in \cite{Ka}, is given by 
\[
(H^*(BH; \Q) \ox H^*(G; \Q), D) = (H^*(BH; \Q) \ox \wedge (x_1, \dots, x_{r_G}), D),
\]
with the differential $D$ determined by
\[
D|_{H^*(BH, \Q)} \equiv 0 \ \textrm{ and }\ D(x_i) = (B_f)^*(\bar x_i \ox 1) - (B_f)^*(1 \ox \bar x_i),
\]
where $(B_f)^* : H^*(BG; \Q) \ox H^*(BG;\Q) \to H^*(BH; \Q)$ is the map induced by the (injective) homomorphism $f:H \to G \x G$ which describes the free action of $H$ on $G$.  In order to compute the map $(B_f)^*$, one need only follow the procedure as laid out in \cite{Es} (for further explicit examples, see \cite{FZ}, \cite{Ke}, \cite{DV}).

In the present situation, $G = \prod_{i=1}^N \sph^3$ and $H = T^2$, hence $H^*(G; \Q) = \wedge(x_1, \dots, x_N)$, with $\deg(x_i) = 3$ for all $i = 1, \dots, N$, and $H^*(BH; \Q) = \Q[s_1, s_2]$, with $\deg(s_1) = \deg(s_2) =  2$.  Moreover, the map $(B_f)^*$ is determined by
\begin{align*}
(B_f)^*(\bar x_i \ox 1) &= \frac{1}{4}\left((a_i + b_i) s_1 + (k_i + l_i) s_2  \right)^2 \ \textrm{ and } \\
(B_f)^*(1 \ox \bar x_i) &= \frac{1}{4}\left((b_i - a_i) s_1 + (l_i - k_i) s_2\right)^2.
\end{align*}
It now follows easily that the minimal model for $(\prod_{i=1}^N \sph^3) \bq T^2$ is given by
\[
(\Q[s_1, s_2] \ox \wedge(x_1, \dots, x_N), D)
\]
where $D(s_1) = D(s_2) = 0$ and 
\begin{align*}
D(x_i) &= (a_i s_1 + k_i s_2)(b_i s_1 + l_i s_2) \\ 
&= a_i b_i s_1^2 + (a_i l_i + b_i k_i) s_1 s_2 + k_i l_i s_2^2
\end{align*}
for all $i = 1, \dots, N$.

By Lemma \ref{L:reduced}, it may be assumed without loss of generality that $a_1 \neq 0$, $k_1 = 0$, $(b_1, l_1) \neq (0,0)$ and $k_2 l_2 \neq 0$. By rescaling the $x_i$ appropriately, it can be further assumed that $a_1=1$ and $\gcd(b_1,l_1)=1$. Under these assumptions the matrix associated to the map
\[
D_3:  \Span_\Q \{ x_1, \dots, x_N \} = \Q^N  \to \Q^3 = \Span_\Q \{ s_1^2, s_1s_2, s_2^2 \} = H^4(BH;\Q)
\]
is the one that appears in Lemma \ref{L:same}, and, in particular, its image has dimension at least $2$.

If $D_3$ has a three-dimensional image, then there is a unique minimal model and hence a unique rational homotopy type, since there is always some basis $\{y_1, \dots, y_N\}$ for $H^3 (G;\Q) = \Q^N$, with $N \geq 3$, such that
\begin{align*}
D_3(y_1) &= s_1^2, \\
D_3(y_2) &= s_1 s_2, \\
D_3(y_3) &= s_2^2, \\
D_3(y_j) &= 0, \ \textrm{ for all } j = 4, \dots, N.
\end{align*}
An action achieving this model is given by setting $a_1 = b_1 = 1$, $k_1 = l_1 = 0$, $a_2 = b_2 = 0$, $k_2 = l_2 = 1$, $a_3 = l_3 = 2$, $b_3 = k_3 = 0$ and $a_i = b_i = k_i = l_i = 0$, for all $i = 4, \dots, N$.  The corresponding biquotient $(\prod_{i=1}^N \sph^3) \bq T^2$ is the product $T^1 (\sph^2 \x \sph^2) \x \prod_{i = 1}^{N-3} \sph^3$.  Indeed, $T^1 (\sph^2 \x \sph^2)$ is given as the quotient $(\sph^3 \x \sph^3 \x \sph^3) \bq T^2$, where $T^2$ acts via
\[
(z,w) \cdot \bpm q_1 \\ q_2 \\ q_3 \epm = 
\bpm 
z q_1 \\ 
w q_2 \\ 
z^2 u_3 +  w^2 v_3 j 
\epm , 
\]
where $q_3 = u_3 + v_3 j  \in \sph^3 \subset \HH$ as usual.  One sees this as follows: The projection onto the first two $\sph^3$ factors shows that this is an $\sph^3$-bundle over $\sph^2 \x \sph^2$.  The associated vector bundle $E$ is the quotient of $\sph^3 \x \sph^3 \x \HH$ by the $T^2$ action described above and it suffices to show that $E$ is the tangent bundle of $\sph^2 \x \sph^2$.  By considering the $z$- and $w$-circle actions separately, it is clear, however, that $E = (\sph^3 \x \C)/\sph^1 \x (\sph^3 \x \C)/\sph^1$, where the Euler class shows that each factor is $T \sph^2$.

\medskip
It remains to consider the case where $D_3$ has a two-dimensional image. Given $a_1 = 1$ and $\gcd(b_1, l_1) = 1$, consider the system of equations
\beq
\label{Eq:system}
\begin{aligned}
D_3(x_1) &= b_1 s_1^2 + l_1 s_1 s_2,\\
D_3(x_i) &= a_j b_j s_1^2 + (a_j l_j + b_j k_j) s_1 s_2 + k_j l_j s_2^2,\ \textrm{ for all } j = 2, \dots, N.
\end{aligned}
\eeq
If $l_1 = 0$, it follows that $b_1 = \pm 1$.  By subtracting an appropriate multiple of $x_1$ from $x_2$ and, by an abuse of notation, relabelling the result $x_2$, one achieves a differential as in the hypothesis of Lemma \ref{L:easy}.  After applying the lemma, it may be assumed without loss of generality that $D_3(x_1) = s_1^2$ and $D_3(x_2) = s_2^2$.  Since all other terms in the image of $D_3$ are linear combinations of $D_3(x_1)$ and $D_3(x_2)$, an appropriate change of basis yields, again abusing notation, $D_3(x_1) = s_1^2$, $D_3(x_2) = s_2^2$, and $D_3(x_j) = 0$ for all $j = 3, \dots, N$.  The resulting minimal model is that of $(\sph^2 \x \sph^2) \x \prod_{i = 1}^{N-2} \sph^3$.

Suppose now that $l_1 \neq 0$. Set 
$\tilde s_2 = b_1 s_1 + l_1 s_2$, hence $s_2 = \tfrac{1}{ l_1}(\tilde s_2 -  b_1 s_1)$.  Therefore
\begin{align*}
D_3(x_1)
&=  s_1 \tilde s_2,\\
D_3(x_j) &= \left(a_j s_1 + \frac{k_j}{l_1}(\tilde s_2 - b_1 s_1) \right)
		\left(b_j s_1 + \frac{l_j}{l_1}(\tilde s_2 - b_1 s_1) \right) \\
&= l_1^2 \left(- \bvm b_1 & a_j \\ l_1 & k_j \evm s_1 + k_j \tilde s_2 \right)
			\left( - \bvm b_1 & b_j \\ l_1 & l_j \evm s_1 + l_j \tilde s_2 \right),
\end{align*}
for all $j = 2, \dots, N$.  Finally, if $\tilde x_j$, $j = 2, \dots, N$, is defined by
\[
\tilde x_j = \frac{1}{l_1^2} x_j +   
\left(l_j \bvm b_1 & a_j \\ l_1 & k_j \evm + k_j 
		\bvm b_1 & b_j \\ l_1 & l_j \evm \right) x_1
\] 
then, using the linearity of the determinant function in the first column, the system of equations reduces to
\begin{align*}
D_3(x_1) 
&= s_1 \tilde s_2,\\
D_3(\tilde x_j) &= \bvm b_1 & a_j \\ l_1 & k_j \evm  
\bvm b_1 & b_j \\ l_1 & l_j \evm s_1^2 + k_j l_j \tilde s_2^2,
\end{align*}
for all $j = 2, \dots, N$.

By Lemma~\ref{L:same}, it follows that there is some $\varepsilon\in \{\pm 1\}$ such that
\[
\bvm b_1 & a_j \\ l_1 & k_j \evm  
\bvm b_1 & b_j \\ l_1 & l_j \evm = \varepsilon k_j l_j, \text{ for all } j = 2, \dots, N. 
\]
As $k_2 l_2 \neq 0$ and the image of $D_3$ is two dimensional, let $\tilde x_2'$ be the appropriate rescaling of $\tilde x_2$, and $\tilde x_j'$ be the relevant linear combinations of $x_1$ and $\tilde x_2'$, such that the differential $D$ can be written as
\begin{align*}
D(x_1) 
&= s_1 \tilde s_2,\\
D(\tilde x_2') &= s_1^2 \pm \tilde s_2^2,\\
D(\tilde x_j') &= 0, \ \textrm{ for all } j = 3, \dots, N. 
\end{align*}

Lemma~\ref{L:easy} shows that, when $D(\tilde x_2') = s_1^2 + \tilde s_2^2$, the resulting minimal model is that of $(\sph^2 \x \sph^2) \x \prod_{i = 1}^{N-2} \sph^3$.  On the other hand, whenever $D(\tilde x_2') = s_1^2 - \tilde s_2^2$, the minimal model corresponds to that of $(\CP^2 \# \CP^2) \x \prod_{i = 1}^{N-2} \sph^3$.
 \end{proof}

% SECTION: DIFFEO CLASSIFICATION
%-----------------------------------------------------
%-----------------------------------------------------

\section{Partial classification in low dimensions}
\label{S:DIFFEO_CLASS}

In low dimensions, the classification in Theorem \ref{T:RIGIDITY} can be significantly strengthened.  If $M^3$ is a smooth, closed, simply connected, rationally elliptic manifold of dimension three, then, by the Poincar\'e Conjecture, $M^3$ is diffeomorphic to $\sph^3$ and admits a unique free $\sph^1$ action, the so-called Hopf action, and infinitely many almost-free $\sph^1$ actions (see, for example, \cite{Or}).  Moreover, as there is a unique effective $T^2$ action on $\sph^3$ (see \cite{Ne}), the classification of effective torus actions up to equivariant diffeomorphism is complete.

A classification up to homeomorphism of closed, (simply connected) rationally elliptic $4$-manifolds can be found in \cite{PP}, with the complete list consisting of the spaces $\sph^4$, $\CP^2$, $\sph^2 \x \sph^2$ and $\CP^2 \# \pm \CP^2$.  This can be improved to (equivariant) diffeomorphism in the presence of a smooth circle action by employing a result of Fintushel \cite[Theorem 13.2]{Fin} combined with the Poincar\'e Conjecture.  By Proposition \ref{P:ELLIPTIC_INEQ}, none of these $4$-manifolds can admit an almost-free $\sph^1$ action.  On the other hand, since a maximal effective torus action is of rank two (i.e.\ of cohomogeneity two), the classification of such actions up to equivariant diffeomorphism follows from the results in \cite{GGK} and \cite{GW}.

Closed, simply connected manifolds of dimension five have been classified up to diffeomorphism by Barden \cite{Ba}.  If a closed, simply connected manifold $M^5$ is assumed to be rationally elliptic, then Proposition \ref{P:ELLIPTIC_INEQ} can be used to determine the rational homotopy groups and, hence, the minimal model and rational cohomology ring for $M^5$.  It follows that $M^5$ is either a rational homology $5$-sphere or has Betti numbers $b_2(M^5) = b_3(M^5) = 1$.  From Barden's classification, it is clear that there are infinitely many possible diffeomorphism types.  If $M^5$ admits, in addition, a free $\sph^1$ action, then the quotient $B^4 = M^5/\sph^1$ is a closed, simply connected, rationally elliptic $4$-manifold with $1 \leq \rank(\pi_2(B^4)) \leq 2$, hence is homeomorphic to one of $\CP^2$, $\sph^2 \x \sph^2$ or $\CP^2 \# \pm \CP^2$.  Since $M^5$ is simply connected, the Gysin sequence and \cite{Ba} together yield that $M^5$ is diffeomorphic to one of $\sph^5$, $\sph^3 \x \sph^2$ or $\sph^3 \tilde\x \sph^2$, the non-trivial $\sph^3$-bundle over $\sph^2$.  If the circle action on $M^5$ is assumed to be only almost free, the classification result of Koll\'ar \cite{Ko} describes which $5$-manifolds arise.  In particular, there can be torsion, albeit strongly restricted, in the cohomology ring.

If the rationally elliptic manifold $M^5$ admits a maximal effective torus action, that is, a torus action of rank three, then a combination of the work of Oh \cite{Oh2} with the classification in \cite{Ba} yields that $M^5$ must again be diffeomorphic to one of $\sph^5$, $\sph^3 \x \sph^2$ or $\sph^3 \tilde\x \sph^2$.  Moreover, the results in \cite{GGK} give a classification of such actions up to equivariant diffeomorphism.

In dimension six, closed, simply connected manifolds have been classified by Wall \cite{Wa}, Jupp \cite{Ju} and Zhubr \cite{Zh}.  In particular, every closed, simply connected $6$-manifold $M^6$ is diffeomorphic to a connected sum of the form $M_0^6 \# M_1^6$, where $H_3(M_0^6; \Z)$ is finite and $M_1^6$ is a connected sum of copies of $\sph^3 \x \sph^3$.  If $M^6$ is rationally elliptic and admits an almost-free $T^2$ action (in fact, an almost-free circle action is sufficient), then one can easily determine from Proposition \ref{P:ELLIPTIC_INEQ} that $M^6$ has Betti numbers $b_2(M^6) = 0$ and $b_3(M^6) = 2$, that is, $M^6 \cong M_0^6 \# (\sph^3 \x \sph^3)$, where $M_0^6$ is a rational homology $6$-sphere.  It is not clear which such $M^6$ admit an almost-free $T^2$ action.  However, if the $T^2$ action on $M^6$ is free, then, being the total space of a principal bundle over a closed, (simply connected) rationally elliptic $4$-manifold with $b_2(M^6/T^2) = 2$, it turns out that $M^6$ is homeomorphic, hence diffeomorphic, to $\sph^3 \x \sph^3$.

On the other hand, the case where $M^6$ admits an effective $T^4$ action is very rigid.  Indeed, it follows from \cite{Oh} that $M^6$ is equivariantly diffeomorphic to $\sph^3 \x \sph^3$ equipped with its unique smooth, effective $T^4$ action.

In dimensions $7$ to $9$, it is also possible to obtain a classification in some special cases, although a general classification seems out of reach at present.  Nevertheless, Theorem \ref{T:7to9} below provides further evidence for the conjecture in the introduction.  First, using the notation established in Section \ref{S:PROOF_THM_A}, recall that the proofs of Theorem \ref{T:RANK_BOUND} and Theorem \ref{T:RIGIDITY} yield $s = n - k$ whenever $k = \left\lfloor \frac{2n}{3} \right\rfloor$.  Thus $M^n$ admits an almost-free action by a subtorus of rank $k - s = 2k-n$.

% THM

\begin{thm}
\label{T:7to9}
Let $M^n$ be a smooth, closed,  (simply connected) rationally elliptic $n$-dimensional manifold, $7 \leq n \leq 9$, equipped with a smooth, effective action of the torus $T^k$ of rank $k = \left\lfloor \frac{2n}{3} \right\rfloor$.  Suppose further that $H_2(M^n; \Z)$ is torsion free and that $T^k$ contains a subtorus of rank $2k-n$ which acts freely on $M^n$.  Then the action of $T^k$ on $M^n$ is equivariantly homeomorphic to the unique (induced) effective, linear action of $T^k$ on a manifold of one of the following forms:
\begin{align*}
n = 7 : & \begin{cases}
\sph^7 \text{ or } \ \sph^4 \x \sph^3, & \text{if } b_2(M^7) = 0; \\
(\sph^3 \x \sph^5)/\sph^1, & \text{if } b_2(M^7) = 1; \\
(\sph^3 \x \sph^3 \x \sph^3)/ T^2, & \text{if } b_2(M^7) = 2.
\end{cases} \\
n = 8 : & \begin{cases}
\sph^3 \x \sph^5, & \text{if } b_2(M^8) = 0; \\
(\sph^3 \x \sph^3 \x \sph^3)/ \sph^1, & \text{if } b_2(M^8) = 1.
\end{cases} \\
n = 9 : & \quad \sph^3 \x \sph^3 \x \sph^3.
\end{align*}
\end{thm}

% PROOF

\begin{proof}
First note that, as $7 \leq n \leq 9$ and $k = \left\lfloor \frac{2n}{3} \right\rfloor$, it follows that $n - k = 3$. Now, let $T^{2k-n} \In T^k$ be a subtorus acting freely on $M^n$ and let $B^6 = M^n/T^{2k-n}$ be the corresponding quotient.  In particular, there is an induced effective $T^3 = T^k/T^{2k-n}$ action on $B^6$.  From the long exact homotopy sequence for the principal bundle $T^{2k-n} \to M^n \to B^6$ it follows that $\pi_1(B^6) = 0$ and $\pi_2(B^6) = \pi_2(M^n) \oplus \Z^{2k-n}$.  As $H_2(M^n; \Z)$ is torsion free, one obtains $H_2 (B^6; \Z) = \Z^{b_2(M^n) + 2k - n}$, by applying the Hurewicz Theorem first to $M^n$ and then to $B^6$.  The Universal Coefficient Theorem, together with Poincar\'e Duality, now yields $H^1(B^6; \Z) = H^5(B^6; \Z) = 0$, $H^2(B^6; \Z) = H^4(B^6; \Z) = \Z^{b_2(M^n) + 2k - n}$ and that $H^3(B^6; \Z)$ is torsion free.

Given as before $d_j(X) = \dim(\pi_j (X) \ox \Q)$ for a space $X$, it can easily be seen from the long exact homotopy sequence for $T^{2k-n} \to M^n \to B^6$ that $d_2(B^6) = d_2(M^n) + 2k - n$ and $d_j(B^6) = d_j(M^n)$, for all $j \geq 3$.  In particular, $B^6$ is rationally elliptic and, from the values of $d_j(M^n)$ determined in Lemmas \ref{L:d_even} and \ref{L:RHG}, as well as the proof of Theorem \ref{T:ThmB2}, one obtains 
$$
\chi_\pi (B^6) = \sum_{j = 0}^\infty (-1)^j \, d_j(B^6) = \chi_\pi(M^n) - (2k - n) = 0.
$$
This identity has a number of implications, see \cite[Prop.\ 32.10]{FHT}.  First, $H^{\rm odd}(B^6; \Q) = 0$ and, together with the discussion above, this implies that $H^{\rm odd}(B^6; \Z) = 0$.  Second, the Euler characteristic $\chi (B^6)$ is positive and, hence, the induced effective $T^3$ action on $B^6$ must have fixed points.  Consequently, $B^6$ is a (simply connected) rationally elliptic, torus manifold with $H^{\rm odd}(B^6; \Z) = 0$.

By \cite{Wi}, $B^6$ is therefore homeomorphic to the quotient of a product $\prod_{i=1}^m \sph^{k_i}$, $k_i \geq 3$, by a free, linear action of the torus $T^r$ of rank $r = \#\{i \mid k_i \text{ odd}\}$.  In combination with $\pi_2(B^6) = \Z^{b_2(M^n) + 2k - n}$, the long exact homotopy sequence of the principal bundle $T^r \to \prod_{i=1}^m \sph^{k_i} \to B^6$ now yields that $r =b_2(M^n) + 2k - n$.  As there is a unique principal $T^r$-bundle over $B^6$ with $2$-connected total space, it follows that $M^n$ must be homeomorphic to the quotient of $\prod_{i=1}^m \sph^{k_i}$ by a free, linear $T^{b_2(M^n)}$ action.

Now, in the proof of Theorem \ref{T:RIGIDITY} it was shown that $d_2(M^n) = b_2(M^n) \in \{0, 1, 2\}$,with restrictions depending on $n$, and the possible values of the $k_i$ were determined in each case, as these follow from the possible values of $d_j(M^n)$.  Hence, $M^n$ must be homeomorphic to a manifold of one of the forms listed in the statement of the theorem.

Finally, the equivariance of the homeomorphism follows from \cite{Wi} together with the uniqueness of maximal-rank, linear actions on products of spheres.
\end{proof}

As an interesting and illustrative example, the Lie group $\SU(3)$ is rationally homotopy equivalent to $\sph^3 \x \sph^5$, but $\pi_4$ shows that they are not even homotopy equivalent, never mind homeomorphic.  Given that there exist (at least two, see \cite{Es}) free torus actions on $\SU(3)$ of rank $\left\lfloor \frac{8}{3} \right\rfloor = 2$, Theorem \ref{T:7to9} states that such an action cannot be extended to a smooth, effective torus action of rank $\left\lfloor \frac{16}{3} \right\rfloor = 5$, even though there are extensions to $T^4$ actions. It is expected that $\SU(3)$ does not admit any smooth, effective $T^5$ actions whatsoever.

% REMARK

\begin{rem}
(a) There are several articles dealing with the classification up to diffeomorphism of the manifolds which appear in the conclusion of Theorem  \ref{T:7to9}.  See, for example, \cite{DV2, CE, Kr}.

\noindent 
(b) The difficulty in extending Theorem \ref{T:7to9} to higher dimensions lies in establishing that $H^*(B^{2(n-k)}; \Z)$ has no torsion in odd degrees.  This is essential in order to apply the results in \cite{Wi} in the case that $M^n$ is rationally elliptic.  On the other hand, by assuming  in \cite{ES} that $M^n$ possesses instead an invariant metric of non-negative curvature, the authors avoid this issue entirely.  In general, it is unclear how to proceed if the $T^{2k-n}$ action on $M^n$ is only almost free.
\end{rem}

%--------------------------------------------------------------------------
% 				BIBLIOGRAPHY
%--------------------------------------------------------------------------

\bibliographystyle{amsplain}

\begin{thebibliography}{10}

\bibitem{Al} C.\ Allday, \emph{On the rank of a space}, Trans.\ Amer.\ Math.\ Soc. {\bf 166} (1972), 173--185.

\bibitem{AP} C.\ Allday~and V.\ Puppe,  \emph{Cohomological methods in transformation groups}, Cambridge Studies in Advanced Mathematics, 32. Cambridge University Press, Cambridge, 1993.

\bibitem{Ba} D.\ Barden, \emph{Simply connected five-manifolds}, Ann.\ Math. \textbf{82} (1965), 365--385.

%\bibitem{Br} G.\ E.\ Bredon, \emph{Introduction to compact transformation groups}, Pure and Applied Mathematics, 46, Academic Press, 1972. 

\bibitem{DV} J.\ DeVito, \emph{The classification of compact simply connected biquotients in dimensions $4$ and $5$}, Differential Geom.\ Appl. {\bf 34} (2014), 128--138.

\bibitem{DV2} J.\ DeVito, \emph{The classification of compact simply connected biquotients in dimension $6$ and $7$}, Math.\ Ann. \textbf{368} (2017), no. 3-4, 1493--1541.

\bibitem{Es} J.-H.\ Eschenburg, \emph{Cohomology of biquotients}, Manuscripta Math. {\bf 75} (1992), 151--166.

\bibitem{CE} C.\ Escher, \emph{A diffeomorphism classification of generalized Witten manifolds}, Geom.\ Dedicata {\bf 115} (2005), 79--120.

\bibitem{ES} C.\ Escher and C.\ Searle, \emph{Non-negative curvature and torus actions}, preprint 2015, arXiv:1506.08685.

\bibitem{FHT} Y.\ F\'elix, S.\ Halperin~and J.\ Thomas, \emph{Rational homotopy theory}, Graduate Texts in Mathematics, 205. Springer-Verlag, New York, 2001. 

\bibitem{FOT} Y.\ F\'elix, J.\ Oprea~and D.\ Tanr\'e, \emph{Algebraic models in geometry}, Oxford Graduate Texts in Mathematics, 17. Oxford University Press, Oxford, 2008. 

\bibitem{Fin} R.\ Fintushel, \emph{Classification of circle actions on $4$-manifolds}, Trans.\ Amer.\ Math.\ Soc. \textbf{242} (1978), 377--390.

\bibitem{FZ} L.\ Florit and W.\ Ziller, \emph{On the topology of positively curved Bazaikin spaces}, J. Europ. Math. Soc. {\bf 11} (2009), 189--205.

\bibitem{GGK} F.\ Galaz-Garc\'ia and M.\ Kerin, \emph{Cohomogeneity-two torus actions on nonnegatively curved manifolds of low dimension}, Math. Z. {\bf 276} (2014), 133--152.

\bibitem{GKRW} F.\ Galaz-Garc\'ia, M.\ Kerin, M.\ Radeschi~and M.\ Wiemeler, \emph{Torus orbifolds, slice-maximal torus actions and rational ellipticity}, Int. Math. Res. Not. IMRN  2018, no.\ 18, 5786--5822.

\bibitem{GGS} F.\ Galaz-Garc\'ia and C.\ Searle, \emph{Low-dimensional manifolds with non-negative curvature and maximal symmetry rank}, Proc.\ Amer.\ Math.\ Soc. {\bf 139} (2011), 2559--2564.

\bibitem{GH} K.\ Grove and S.\ Halperin, \emph{Dupin hypersurfaces, group actions and the double mapping cylinder}, J.\ Differential Geom. {\bf 26} (1987), 429--459.

\bibitem{GW} K.\ Grove and B.\ Wilking, \emph{A knot characterization and $1$-connected nonnegatively curved $4$-manifolds with circle symmetry}, Geom.\ Top. {\bf 18} (2014), 3091--3110.

\bibitem{Ha} S.\ Halperin, \emph{Rational homotopy and torus actions}, Aspects of Topology, London Math.\ Soc. Lecture Note Series {\bf 93} (1985), 293--306.

\bibitem{Is} H.~Ishida, \emph{Complex manifolds with maximal torus actions}, preprint 2015, arXiv:1302.0633v3.

\bibitem{Ju} P.\ Jupp, \emph{Classification of certain $6$-manifolds}, Proc.\ Camb.\ Philos.\ Soc. \textbf{73} (1973), 293--300.

\bibitem{Ka} V.\ Kapovitch, \emph{A note on rational homotopy of biquotients}, preprint, unpublished,  \url{http://www.math.toronto.edu/vtk/biquotient.pdf}.

\bibitem{Ke} M.\ Kerin, \emph{Some new examples with almost positive curvature}, Geom.\ Top.\ {\bf 15} (2011), 217--260.

\bibitem{Ko} J.\ Koll\'ar, \emph{Circle actions on simply connected $5$-manifolds}, Topology \textbf{45} (2006), 643--671.

\bibitem{Kr} B.\ Kruggel, \emph{A homotopy classification of certain $7$-manifolds}, Trans.\ Amer.\ Math.\ Soc. {\bf 349} (1997), 2827--2843.

\bibitem{Ne} W.\ D.\ Neumann, \emph{$3$-dimensional $G$-manifolds with $2$-dimensional orbits}, 1968 Proc.\ Conf.\ on Transformation Groups (New Orleans, La., 1967) pp. 220--222 Springer, New York 

\bibitem{Oh} H.\ S.\ Oh, \emph{$6$-dimensional manifolds with effective $T^4$-actions}, Topology Appl. {\bf 13} (1982), 137--154.

\bibitem{Oh2} H.\ S.\ Oh, \emph{Toral actions on $5$-manifolds}, Trans.\ Amer.\ Math.\ Soc. \textbf{278} (1983), 233--252. 

\bibitem{Or} P.\ Orlik, \emph{Seifert manifolds}, Lecture Notes in Mathematics, Vol.\ 291. Springer-Verlag, Berlin-New York, 1972. 

\bibitem{PP} G.\ Paternain and J.\ Petean, \emph{Minimal entropy and collapsing with curvature bounded from below}, Invent.\ Math. {\bf 151} (2003), 415--450.

\bibitem{To} B.\ Totaro, \emph{Curvature, diameter, and quotient manifolds}, Math.\ Res.\ Lett. {\bf 10} (2003), 191--203.

\bibitem{Us} Yu.\ M. Ustinovsky, \emph{Geometry of compact complex manifolds with maximal torus action}, Tr. Mat. Inst. Steklova {\bf 286} (2014), 219--230;  translation in Proc. Steklov Inst. Math. {\bf 286} (2014), 198--208.

\bibitem{Wa} C.\ T.\ C.\ Wall, \emph{Classification problems in differential topology. V: On certain $6$-manifolds}, Invent. Math. {\bf 1} (1966), 355--374.

\bibitem{Wi} M.\ Wiemeler, \emph{Torus manifolds and non-negative curvature}, J.\ London Math.\ Soc. {\bf 91} (2015), 667--692.

\bibitem{Zh} A.\ V.\ Zhubr, \emph{Closed simply connected six-dimensional manifolds: proofs of classification theorems}, Algebra i Analiz \textbf{12} (2000), 126--230.
\end{thebibliography}

% ----------------------------------------------------------------
% END BIBLIOGRAPHY-------------------------------------------------
% ----------------------------------------------------------------

\end{document}